\documentclass[11pt]{article}
\usepackage[numbers,sort&compress]{natbib}
\usepackage{enumerate}
\usepackage{tikz}
\usepackage{pgfplots}
\usepackage{amscd}
\usepackage{amsmath}
\usepackage{array}
\usepackage{latexsym}
\usepackage{amsfonts}
\usepackage{amssymb}
\usepackage{amsthm}
\usepackage{verbatim}
\usepackage{mathrsfs}
\usepackage{enumerate}
\usepackage{hyperref}
\newcommand{\weixiao}{\fontsize{0.7pt}{\baselineskip}\selectfont}

\theoremstyle{plain}\newtheorem{definition}{Definition}[section]
\theoremstyle{definition}\newtheorem{theorem}{Theorem}[section]
\theoremstyle{plain}\newtheorem{lemma}[theorem]{Lemma}
\theoremstyle{plain}\newtheorem{coro}[theorem]{Corollary}
\theoremstyle{plain}\newtheorem{prop}[theorem]{Proposition}
\theoremstyle{remark}\newtheorem{remark}{Remark}[section]
\usepackage{xcolor}

\newcommand{\wgr}[1]{\textcolor{black}{#1}}

\newcommand{\wred}[1]{\textcolor{black}{#1}}

\newcommand{\Div}{\mathrm{div}\,}
\newcommand{\B}{\Big}

\newcommand{\R}{\mathbb{R}}
\newcommand{\be}{\begin{equation}}
\newcommand{\ee}{\end{equation}}
\newcommand{\ba}{\begin{aligned}}
	\newcommand{\ea}{\end{aligned}}

\providecommand{\bysame}{\leavevmode\hbox to3em{\hrulefill}\thinspace}
\newcommand{\f}{\frac}

\newcommand{\ben}{\begin{enumerate}}
	\newcommand{\een}{\end{enumerate}}

\newcommand{\ti}{\nabla}

\newcommand{\Rmnum}[1]{\expandafter\@slowromancap\romannumeral #1@}

\allowdisplaybreaks

\numberwithin{equation}{section}
\begin{document}
	%%%%%%%%%%%%%%%%%%%%%%%%%%%%%%%%%%%%%%%%%%%%%%%%%%%%%%%%%%%%%%%%%%%%%%%%%%%%%%%%%%%%%%%%%%%%%%%%%%%%
\title{   Leray's backward self-similar solutions to
		   the 3D Navier-Stokes equations in Morrey spaces}
		\author{Quansen Jiu\footnote{
School of Mathematical Sciences, Capital Normal University, Beijing 100048, P. R. China Email: jiuqs@cnu.edu.cn},\,\,\,\, Yanqing Wang\footnote{ Department of Mathematics and Information Science, Zhengzhou University of Light Industry, Zhengzhou, Henan  450002,  P. R. China Email: wangyanqing20056@gmail.com}\;~ and\, \, Wei Wei\footnote{Center for Nonlinear Studies, School of Mathematics, Northwest University, Xi'an, Shaanxi 710127, P. R. China Email: ww5998198@126.com } }
\date{}
\maketitle
\begin{abstract}
In this paper,
it is shown that there
  does not exist a non-trivial Leray's backward self-similar solution to the 3D Navier-Stokes equations with profiles in Morrey spaces $\dot{\mathcal{M}}^{q,1}(\mathbb{R}^{3})$ provided $3/2<q<6$, or in $\dot{\mathcal{M}}^{q,l}(\mathbb{R}^{3})$ provided $6\leq q<\infty$ and $2<l\leq q$.
  This  generalizes the   corresponding
results obtained by
  Ne\v{c}as-R\r{a}u\v{z}i\v{c}ka-\v{S}ver\'{a}k
   \cite[Acta. Math. 176 (1996)]{[NRS]} in $L^{3}(\mathbb{R}^{3})$, Tsai \cite[Arch. Ration. Mech. Anal. 143 (1998)]{[Tsai]} in $L^{p}(\mathbb{R}^{3})$ with $p\geq3$,
  Chae-Wolf \cite[Arch. Ration. Mech. Anal. 225 (2017)]{[CW]} in Lorentz spaces $L^{p,\infty}(\mathbb{R}^{3})$ with $p>3/2$, and Guevara-Phuc
\cite[SIAM J. Math. Anal. 12 (2018)]{[GP2]} in $\dot{\mathcal{M}}^{q,\f{12-2q}{3}}(\mathbb{R}^{3})$ with $12/5\leq q<3$ and in $L^{q, \infty}(\R^3)$ with $12/5\leq q<6$.
% To this end, we derive a Caccioppoli type inequality just in terms of velocity field $u$ in Morrey spaces and
 % new energy bound in terms of both the velocity $u$ and the pressure $\pi$ in Morrey spaces.  As a by-product, we present two types of $\varepsilon$-regularity %  criteria at one scale in Morrey spaces, which is of independent interest.
  \end{abstract}
	\noindent {\bf MSC(2000):}\quad 35B65, 35D30, 76D05 \\\noindent
	{\bf Keywords:} Navier-Stokes equations;   self-similar solutions; Morrey spaces
	%%%%%%%%%%
	\section{Introduction}
	\label{intro}
	\setcounter{section}{1}\setcounter{equation}{0}
		We study  the following   incompressible Navier-Stokes equations in   three-dimensional space
	\be\left\{\ba\label{NS}
	&u_{t} -\Delta  u+ u\cdot\ti
	u+\nabla \pi=0, ~~\Div u=0,\\
	&u|_{t=0}=u_0,
	\ea\right.\ee
	where $u $ stands for the flow  velocity field, the scalar function $\pi$ represents the   pressure.   The
	initial  velocity $u_0$ satisfies   $\text{div}\,u_0=0$.

In a seminal work \cite{[Leray]}, Leray introduced the backward self-similar solutions to construct  singular solutions to the 3D Navier-Stokes equations \eqref{NS}. The so-called backward self-similar solution is a weak solution $(u,\pi)$  of \eqref{NS} satisfying, for $a>0, T\in \mathbb{R},$
\be\ba
&u(x,t)=\wred{\f{1}{\sqrt{2a(T-t)}}U\bigg(\f{x}{\sqrt{2a(T-t)}}\bigg)}, \label{learysb}\\
&\pi(x,t)=\f{1}{ 2a(T-t)}\Pi\bigg(\f{x}{\sqrt{2a(T-t)}}\bigg),
\ea\ee
where $U=(U_{1},U_{2},U_{3})$ and $\Pi$ are defined in $\mathbb{R}^{3}$,
and the pair $(u(x,t),\pi(x,t))$ is defined in $\mathbb{R}^{3}\times(-\infty,T)$.
We obtain a singular solution at $t=T$ if $U\neq0$ and
\be \label{SNS}
	-\Delta  U+a U+a(y\cdot \nabla)U+U\cdot\nabla U+ \nabla \Pi=0  , ~~\Div U=0, ~~y\in \mathbb{R}^{3}.
	 \ee
The first breakthrough of backward self-similar solutions
was due to  Ne{\v{c}}as-R{\r u}{\v{z}}i{\v{c}}ka-{\v S}ver\'ak \cite{[NRS]}. They ruled out the existence of Leray's backward self-similar solutions to the 3D Navier-Stokes system if a weak solution $ U$ of equations \eqref{SNS}  is in $   L^{3}(\mathbb{R}^{3})$ (see Section 2 for the definition of weak solutions of \eqref{SNS}).
Subsequently, various  results involving non-existence of Leray's backward self-similar non-trivial solutions  were  obtained by
 Tsai in \cite{[Tsai]}. Precisely,  he
proved that Leray's backward self-similar solution is trivial under the condition that $U\in L^{p}(\mathbb{R}^{3})$ with $3<p<\infty$, and the solution $U\in L^{\infty}(\mathbb{R}^{3})$ in system \eqref{SNS} is a constant. In addition, Tsai's second result is that
 there does not exist a non-trivial solution of the form \eqref{learysb} if $u$ satisfy
the  local  energy inequality
\begin{equation}\label{local-energy}
\sup_{t<s<T} \int_{B_{x_0}( r)}\dfrac12|u(x, s)|^2dx+\int_{t}^{T}\int_{B_{x_0}( r)} |\nabla u(x,s)|^2dxds<\infty,
\end{equation}
for some ball $B_{x_0}( r)$ and some $t<T$. Here, $B_{x_0}( r)$ denotes the ball of
center $x_0$ and radius $r$.

Very recently, Chae-Wolf \cite{[CW]} and Guevara-Phuc \cite{[GP2]} independently made progress
in this direction. On one hand,
 Chae-Wolf \cite{[CW]} proved that
if $U$ belong to the Lorentz spaces $ L^{q, \infty}(\R^3)$ with $q>\f32$ or
 \be\label{CW1}
\|  U\|_{L^{q}(B_{y_{0}}(1))}+\|\nabla U\|_{L^{2}(B_{y_{0}}(1))} =o(|y_{0}|^{1/2}), ~~\text{as} ~~|y_{0}|\rightarrow\infty,
\ee
then $U$ must be identically zero. Roughly speaking, the proof in \cite{[CW]} relied on  $\varepsilon$-regularity criteria
without pressure at one scale, the decay at infinity in  Lorentz spaces  and Tsai's first result mentioned above.
On the other hand, Guevara-Phuc \cite{[GP2]} showed that there does not  exist a non-trivial solution under the condition that $U$ is in Morrey spaces $ \dot{\mathcal{M}}^{q,\f{12-2q}{3}}(\mathbb{R}^{3})$ with $12/5\leq q<3$, or $U\in L^{q, \infty}(\R^3)$ with $12/5\leq q<6$, or $U\in L^{6}(\mathbb{R}^{3})$.
Their arguments were based on  Riesz potentials in Morrey spaces, Sobolev spaces with negative indices and Tsai's second aforementioned result.
%$\in W^{1,2}_{loc}(\mathbb{R}^{3})$
The definitions of relevant function spaces can be found in Section
2.
Notice that there holds the following embedding relation
\be\label{inclusion} L^{q}(\mathbb{R}^{3})\hookrightarrow L^{q,\infty}(\mathbb{R}^{3})\hookrightarrow \dot{\mathcal{M}}^{q,l}( \mathbb{R}^{3})\hookrightarrow\dot{\mathcal{M}}^{q,1}( \mathbb{R}^{3}),~ 1\leq l<q.\ee
This fact can be found in \cite{[CX]}.

The well-posedness of the 3D Navier-Stokes equations in Morrey spaces was studied in \cite{[MY],[Kato],[Taylor],[CX],[L]}.
Compared with the Lebesgue spaces $L^{q}(\mathbb{R}^{3})$ and  Lorentz spaces $L^{q,\infty}(\mathbb{R}^{3})$, $C_{0}^{\infty}(\mathbb{R}^{3})$ is not dense in
Morrey spaces. It is worth pointing out that Sawano \cite{[Sawano]} showed that $ \dot{\mathcal{M}}^{q,l_{2}}(\mathbb{R}^{n})$  is not dense in $\dot{\mathcal{M}}^{q,l_{1}} (\mathbb{R}^{n})$ if
 $1<l_{1}<l_{2}\leq q$. Based on this, the first objective of this paper is to generalize Guevara and Phuc's result in \cite{[GP2]} to general Morrey spaces. Our result is reformulated as
\begin{theorem}\label{the1.1}
Let $U\in W^{1,2}_{\rm loc}(\R^3)$ be a weak solution  of \eqref{SNS}. If
\be\label{cond1}
U\in \dot{\mathcal{M}}^{q,l}(\mathbb{R}^{3}) ~~ \text{with }~~ 2<l\leq q<\infty,
\ee
  then $U\equiv0$.
\end{theorem}
\begin{remark}
This theorem extends all the known results involving non-existence
of Leray's backward self-similar non-trivial solutions to the 3D
Navier-Stokes equations in \cite{[Tsai],[NRS],[GP2],[CW]}.
\end{remark}

 We follow the path of \cite{[GP2],[Tsai],[NRS]} to prove Theorem \ref{the1.1}. First, in contrast to  \cite{[GP2]}, we construct pressure $\Pi$ in the Morrey space directly under condition \eqref{cond1}
 without the application of
 Riesz potentials in Morrey spaces and dual spaces of Sobolev spaces.
 Second, as \cite{[GP2]}, our main target is   to  deduce  \eqref{local-energy} by deriving the  energy bound \eqref{enebound} in terms of  \eqref{cond1} from the local energy inequality \eqref{loc}.  The key tool in \cite{[GP2]} is to utilize the  duality between $W^{-1,2}(B)$ and the Sobolev space $W^{1,2}_0(B)$ to bound energy flux in local energy inequality \eqref{loc}, namely, for $12/5\leq q<3,$
  \be\ba\label{gpei}
  \left|\int_{T-R^{2}}^{T}\int_{B}(|u|^{2}+2\pi)u\cdot\nabla \phi dxdt\right|
 =&\left|\int_{T-R^{2}}^{T}\langle |u|^{2}+2\pi ,u\cdot\nabla \phi \rangle_{W^{-1,2}(B),W_{0}^{1,2}(B)} dt\right| \\
\leq&C\int_{T-R^{2}}^{T}\lambda^{2-\f6q}(t)\|U\|_{
\mathcal{\dot{M}}^{q,\f{12-2q}{3}}(\mathbb{R}^{3})}^{2} \|\nabla u\|_{L^{2}(B)}dt,
 \ea\ee
 where the Hardy-Littlewood-Sobolev inequality for Riesz potentials in Morrey spaces was used and $\lambda(t)=[2a(T-t)]^{-1/2}$. It is worth remarking that the  upper bound $q<3$  in \cite{[GP2]} comes from the  employment  of Hardy-Littlewood-Sobolev inequality, but this strategy breaks down in the case \eqref{cond1}. Here, we make full use of the Meyer-Gerard-Oru interpolation
inequality \eqref{mgoi} and the fact that Morrey spaces $\dot{\mathcal{M}}^{q,1}(\mathbb{R}^{3})$ can be embedded in Besov spaces $\dot{B}^{-\f{3}{q}}_{\infty,\infty}(\mathbb{R}^{3})$ for $1<q<\infty$. This helps us  control energy flux and the first term in the left hand side of \eqref{gpei} as follows
  $$\ba
&\int_{T-\f{R^{2}}{4}}^{T}\|u\|^{3}_{L^{3}(B(R/2))}dt
\\\leq& C \B(\sup_{T-R^{2}\leq t\leq T}\| u\|^{2}_{L^{2}(B(R))}+\int_{T-R^{2}}^{T}\|\nabla u\|^{2}_{L^{2}(B(R))}dt\B)^{\f32-\f{1}{\alpha}}
\B(\int_{T-\f{R^{2}}{4}}^{T}\lambda^{p-\f{3p}{q}}(t)\|U\|^{p}_{\dot{\mathcal{M}}^{q,1}(\mathbb{R}^{3})}dt\B)^{\f{2}{p\alpha}},
\ea$$
where $\alpha=2/p+3/q<2$.
Additionally, we apply the pressure decomposition  as \cite{[HWZ]}  to bound  the second term in the left hand side of \eqref{gpei}. To sum up, we get the following energy bound
\be\ba \label{enebound}
&\|u\|^2_{L^{\infty}(T-\f{R^{2}}{4},T;L^{2}(B(R/2)))}+\|\nabla u\|^2_{L^{2}(T-\f{R^{2}}{4},T;L^{2}(B(R/2)))} \\ \leq&    CR^{3-2\alpha}\B( \int_{T-R^{2}}^{T}\lambda^{p-\f{3p}q}(t)\|U\|^{p}
_{\mathcal{\dot{M}}^{q,1}(\mathbb{R}^{3})} dt\B) ^{\f2p} +  CR^{\f{6-5\alpha}{2-\alpha}} \B(\int_{T-R^{2}}^{T}\lambda^{p-\f{3p}q}(t)\|U\|^{p}
_{\mathcal{\dot{M}}^{q,1}(\mathbb{R}^{3})} dt\B)^{\f{4}{p(2-\alpha )}}\\& + CR^{1-\f6m} \B(\int_{T-R^{2}}^{T}\lambda^{2-\f{3}m}(t) \|\Pi\|
_{\mathcal{\dot{M}}^{m,1}(\mathbb{R}^{3})} dt \B)^{2}.
\ea\ee
This together with \eqref{local-energy} means Theorem \ref{the1.1}.

It should be mentioned that the extra restriction that $q>2$ and $l>2$ in   \eqref{cond1} resulted from the construction process of pressure $\Pi=\mathcal{R}_i \mathcal{R}_j(U_{i}U_{j})$ in the proof of Theorem \ref{the1.1}. Partially motivated by Chae-Wolf in \cite{[CW]}, our next target is to utilize local suitable weak solutions (see Definition \ref{defilsw}) to remove this restriction.
\begin{theorem}\label{the1.4}
Let $U\in W^{1,2}_{\rm loc}(\R^3)$ be a weak solution  of \eqref{SNS}. If
\be\label{cond1.4}
U\in \dot{\mathcal{M}}^{q,1}(\mathbb{R}^{3}) ~~ \text{with }~~ \f32<q<6,
\ee
  then $U\equiv0$.
\end{theorem}
\begin{remark}
According to \eqref{inclusion}, this theorem is an improvement of corresponding results in \cite{[Tsai],[NRS],[GP2],[CW]}.
\end{remark}
The proof of Theorem \ref{the1.4} is based on a combination of techniques from \cite{[Wolf1],[Tsai],[CW]}. Our starting point is to set up a new Caccioppoli type inequality
 below
\be\ba\label{Caccioppoli inequ}
&\|u\|^{2}_{L^{3}(T-\f{R^{2}}{4},T;L^{\f{18}{5}}(B(\f{R}{2})))}+ \|\nabla u\|^{2}_{L^{2}(T-\f{R^{2}}{4},T;L^{2}(B(\f{R}{2})))}\\
\\
\leq &
  CR^{\f{5q-12}{3q}}\| u\|^{2}_{L^{\f{6q}{2q-3}}(T-R^{2},T;\dot{\mathcal{M}}^{q,1}(B(R)))}
  \\&+CR^{\f{4q-15}{2q-3}}\| u\|^{\f{6q}{2q-3}}_{L^{\f{6q}{2q-3}}(T-R^{2},T;\dot{\mathcal{M}}^{q,1}(B(R)))}
  +CR^{\f{2q-6}{q}}\| u\|^{3}_{L^{\f{6q}{2q-3}}(T-R^{2},T;\dot{\mathcal{M}}^{q,1}(B(R)))}.
\ea\ee
This inequality is derived by local suitable weak solutions to the 3D Navier-Stokes equations \eqref{NS} and
  the aforementioned Meyer-Gerard-Oru interpolation
inequality. Various  Caccioppoli type inequalities were recently established in
 \cite{[JWZ],[CW],[HWZ],[Wolf1],[Wolf2]}.
 All the proofs  rest on
  local suitable weak solutions originated in
  \cite {[Wolf1],[Wolf2]}
  by Wolf.
It is known that any usual suitable weak solution to the Navier-Stokes system enjoys the
local energy inequality \eqref{wloc1} (see \cite[Appendix A, p.1372]{[CW2]}).
The novelty of local suitable weak solutions is that, as stated in \cite{[Wolf1],[Wolf2],[CW]}, the  relevant  local energy
inequality \eqref{wloc1} removed   the non-local effect of the pressure term.
 Based on this, Caccioppoli type inequality \eqref{Caccioppoli inequ} and \eqref{cond1.4}  allow us to derive that
$$
\|\nabla U\|_{L^{2}(B_{y_{0}}(1))}=o(|y_{0}|^{1/2}) ~~~\text{and}~~~
\|  U\|_{L^{3}(B_{y_{0}}(1))}=o(|y_{0}|^{2/3}), ~~\text{as} ~~|y_{0}|\rightarrow\infty.
$$
  However, in light of \eqref{CW1}, this is not enough to show that $U\equiv0$. Notice that \eqref{Caccioppoli inequ} implies $\int_{\mathbb{R}^{3}}|U|^{3}|y|^{-2}dy<\infty$, our approach to overcome this difficulty is to construct the pressure $\Pi$ via $A_p$ weighted inequalities for singular integrals, which enables us to obtain $\pi\in L^{3/2}_{\rm loc}$ in terms of \eqref{learysb}. A similar argument has been used by Tsai in \cite{[Tsai],[Tsai2]} for the proof of his second result \eqref{local-energy}.  Adopting this approach together with the local energy inequality \eqref{wloc1}, we deduce \eqref{local-energy} and finally prove Theorem \ref{the1.4}.

  Roughly, the following four figures summarize known results about non-existence of Leray's
backward self-similar non-trivial solutions in the framework of Morrey spaces $\dot{\mathcal{M}}^{q,p}(\mathbb{R}^{3})$.
%\begin{tikzpicture}[scale=.5,>= stealth]
%\node[above] at (4.2,4.6) {\tiny{$p=q$}};
%\node[below] at (0,-0.05) {\tiny{0}};
%\node[below] at (3.2,-1) {\tiny Figure 1:   Ne{\v{c}}as-R{\r u}{\v{z}}i{\v{c}}ka-{\v S}ver\'ak and Tsai};
%\node[below] at (3,3) {\tiny{(3,3)}};
%\draw[->] (0,0)--(7,0)node[right]{\tiny{$q$}};

%\foreach \x in {1,...,6}
 %    		\draw (\x,0) -- (\x,-.1)
%		node[anchor=north] {\tiny{\x}};
%\draw[->](0,0)--(0,7)node[above]{\tiny{$p$}};
%\foreach \y in {1,...,6}
  %   		\draw (0,\y) -- (.1,\y)
%		node[anchor=east] {\tiny{\y}};

%\draw (3,3)--(6.5,6.5);

%\end{tikzpicture} \vspace{.5cm}
\begin{tikzpicture}[scale=.6,>= stealth]
\node[above] at (4.2,4.6) {\tiny{$p=q$}};
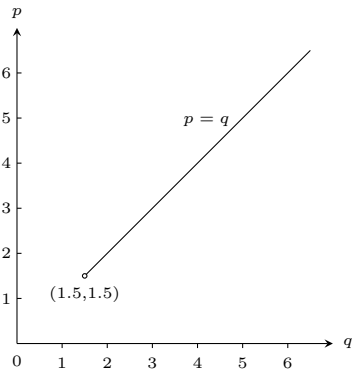
\node[below] at (3.5,-1) {\tiny{Figure 1:  Region of Ne{\v{c}}as-R{\r u}{\v{z}}i{\v{c}}ka-{\v S}ver\'ak, \\ Tsai and Chae-Wolf}};
\draw[->] (0,0)--(7,0)node[right]{\tiny{$q$}};
\foreach \x in {1,...,6}
     		\draw (\x,0) -- (\x,-.1)
		node[anchor=north] {\tiny{\x}};
\draw[->](0,0)--(0,7)node[above]{\tiny{$p$}};
\foreach \y in {1,...,6}
     		\draw (0,\y) -- (.1,\y)
		node[anchor=east] {\tiny{\y}};
\draw (1.5,1.5)--(6.5,6.5);
\draw[fill=white] (1.5,1.5) circle(.05);
\node[below] at (3/2,3/2) {\tiny{(1.5,1.5)}};
\node[below] at (0,-0.05) {\tiny{0}};
\end{tikzpicture} \vspace{.5cm}
\begin{tikzpicture}[scale=.6,>= stealth]
\fill[gray!30] (2.4,2.4) -- ( 3,2)--(3,3);
\draw[->] (0,0)--(7,0)node[right]{\tiny{$q$}};

\foreach \x in {1,...,6}
     		\draw (\x,0) -- (\x,-.1)
		node[anchor=north] {\tiny{\x}};
\draw[->](0,0)--(0,7)node[above]{\tiny{$p$}};
\foreach \y in {1,...,6}
     		\draw (0,\y) -- (.1,\y)
		node[anchor=east] {\tiny{\y}};
\draw[style=dashed] (3,2)--(3,3 );
\draw (2.4,2.4)--(3,2);
\draw (2.4,2.4)--(6,6);
\draw[fill=white] (3,2) circle(.05);
\draw[fill=white] (6,6) circle(.05);
\node[right] at (6,6) {\weixiao{(6,6)}};
\node[below] at (3,2) {\weixiao{(3,2)}};
\node[left] at (12/5,12/5) {\weixiao{(2.4,2.4)}};
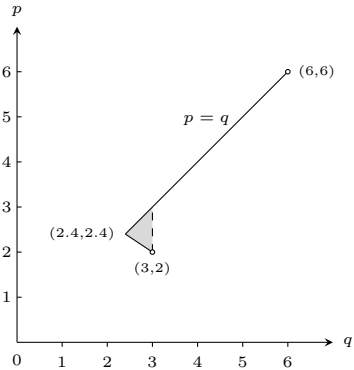
\node[below] at (3.5,-0.8) {\tiny{Figure 2: Region of Guevara-Phuc}};
\node[below] at (0,-0.05) {\tiny{0}};
\node[above] at (4.2,4.6) {\tiny{$p=q$}};
\end{tikzpicture} \vspace{.5cm}

\begin{tikzpicture}[scale=.6,>= stealth]
\node[below] at (2,2) {\weixiao{(2,2)}};
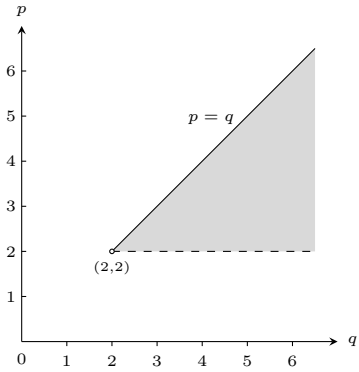
\node[below] at (3.8,-0.8) {\tiny Figure 3: Region of Theorem 1.1};
\fill[gray!30] (2,2) -- ( 6.5,2)--(6.5,6.5);
\draw[->] (0,0)--(7,0)node[right]{\tiny{$q$}};
\node[below] at (0,-0.05) {\tiny{0}};
\foreach \x in {1,...,6}
     		\draw (\x,0) -- (\x,-.1)
		node[anchor=north] {\tiny{\x}};
\draw[->](0,0)--(0,7)node[above]{\tiny{$p$}};
\foreach \y in {1,...,6}
     		\draw (0,\y) -- (.1,\y)
		node[anchor=east] {\tiny{\y}};
\draw[style=dashed] (2,2)--(6.5,2 );
\node[above] at (4.2,4.6) {\tiny{$p=q$}};
\draw (2,2)--(6.5,6.5);
\draw[fill=white] (2,2) circle(.05);
\end{tikzpicture} \vspace{.5cm}
\begin{tikzpicture}[scale=.6,>= stealth]
\node[left] at (1.6,1.5) {\weixiao{(1.5,1.5)}};
\node[right] at (6,6) {\weixiao{(6,6)}};
\node[below] at (6,1) {\weixiao{(6,1)}};
\node[below] at (1.5,1) {\weixiao{(1.5,1)}};
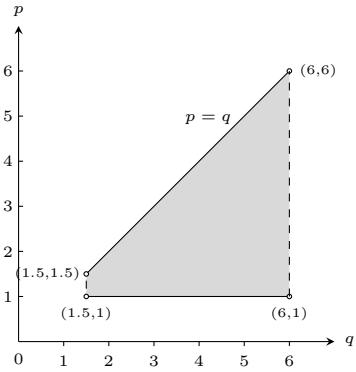
\node[below] at (3.8,-0.8) {\tiny Figure 4: Region of  Theorem 1.2};
\fill[gray!30] (1.5,1.5) -- ( 6,6)-- ( 6,1)--(1.5,1);
\draw[->] (0,0)--(7,0)node[right]{\tiny{$q$}};
\node[below] at (0,-0.05) {\tiny{0}};
\foreach \x in {1,...,6}
     		\draw (\x,0) -- (\x,-.1)
		node[anchor=north] {\tiny{\x}};
\draw[->](0,0)--(0,7)node[above]{\tiny{$p$}};
\foreach \y in {1,...,6}
     		\draw (0,\y) -- (.1,\y)
		node[anchor=east] {\tiny{\y}};

\node[above] at (4.2,4.6) {\tiny{$p=q$}};
\draw (1.5,1.5)--(6,6);
\draw (1.5,1)--(6,1);
\draw[fill=white] (1.5,1.5) circle(.05);
\draw[fill=white] (6,6) circle(.05);
\draw[fill=white] (6,1) circle(.05);
\draw[fill=white] (1.5,1) circle(.05);
\draw[style=dashed] (6,1)--(6,6 );
\draw[style=dashed] (1.5,1.35)--( 1.5,1.15);
\end{tikzpicture}

 %  However,
 %as stated in \cite{[JWZ]}, the cost of local energy inequality \eqref{wloc1} without
%non-local pressure
%is that the    velocity field $u$ is lack of  the  kinetic energy $\|u\|_{L_t^{\infty}L^{2}}$.
Eventually we would like to mention that, as a by-product of the energy bound and the Caccioppoli type inequality obtained in the proof of Theorem 1.1 and Theorem 1.2, one can establish some
  $\varepsilon$-regularity criteria at one scale in Morrey spaces. To the best knowledge of the authors, this is the first $\varepsilon$-regularity criterion involving Morrey spaces, which is of independent interest.  For the details, see Corollary \ref{the1.3} in Section 3 and Corollary  \ref{the1.5} in Section 4.

This  paper is organized as follows. In the  second section,
we recall the definitions of various function spaces and those of suitable weak solutions including local suitable weak solutions. In addition, we present some auxiliary lemmas.
Section 3 is devoted to the proof of Theorem 1.1. To this end, along the line of \cite{[GP2],[Tsai],[NRS]}, we
construct the pressure $\Pi$ in Morrey spaces, and then localize the  Meyer-Gerard-Oru interpolation
inequality. This together with the local energy inequality yields the energy bound, which concludes  the proof of our first theorem.
In Section 4, we deal with the Caccioppoli type inequality by local suitable weak solutions and local Meyer-Gerard-Oru interpolation
inequality obtained in Section 3.  Then $A_p$ weighted inequalities enable us to  recover the pressure. Finally, applying the local energy inequality \eqref{wloc1} again and \eqref{local-energy}, we complete the proof of Theorem 1.2.

 \section{Function spaces and some known facts} \label{sec2}

\subsection{Function spaces}

In this section, we begin with definitions of various function spaces.
Let $\mathcal{S}\left(\mathbb{R}^{n}\right)$ be the set of all Schwartz functions on $\mathbb{R}^{n}$, endowed with the usual topology, and denote by $\mathcal{S}^{\prime}\left(\mathbb{R}^{n}\right)$ its topological dual, namely, the space of all bounded linear functionals on $\mathcal{S}\left(\mathbb{R}^{n}\right)$ endowed with the weak *-topology.
The classical Sobolev space $W^{1,p}(\Omega)$, with $1\leq  p\leq\infty$, is equipped with the norm $\|f\|_{W^{1,p}(\Omega)}= \|f\|_{L^{p}(\Omega)}+ \|\nabla f\|_{L^{p}(\Omega)}$. A function $f\in W^{1,p}_{\rm loc}(\mathbb{R}^{n})$ if and only if $f\in W^{1,p}(B)$ for every ball $B\subset\mathbb{R}^{n}$. The space  $C^{\infty}_{0}(\Omega)$ is  the set of all  the smooth compactly supported functions on $\Omega$.
Let $W_{0}^{1,p}(\Omega)$ be the completion of $C_{0}^{\infty}(\Omega)$ in the norm
of $W^{1,p}(\Omega)$. Let  $W^{-1,p}(\Omega)$ denote the dual of the Sobolev space $W_{0}^{1,p}(\Omega)$.

Next, we present the definitions of Lorentz spaces and Morrey spaces. For $q\in[1,\infty]$, let
$$
L^{q,\infty}(\mathbb{R}^{3})=\big\{f: f~ \text{is  a measurable function  on}~ \mathbb{R}^{3} ~\text{and} ~\|f\|_{L^{q,\infty}(\mathbb{R}^{3})}<\infty\big\}
$$
be the Lorentz space $L^{q,\infty}$ defined by means of the quasinorm
$$
\|f\|_{L^{q,\infty}(\mathbb{R}^{3})}=
\sup_{\alpha>0}\alpha|\{x\in \mathbb{R}^{3}:|f(x)|>\alpha\}|^{\f{1}{q}},
$$
where $|E|$ represents the three-dimensional Lebesgue measure of a set $E\subset \mathbb{R}^{3}$. The Morrey space $\mathcal{\dot{M}}^{p,l}(\Omega)$, with $1\leq l<\infty$, $1\leq  p\leq\infty$  and a domain $\Omega\subset\mathbb{R}^{3}$, is defined as the space of all measurable functions $f$ on $\Omega$ for which the norm
$$
\|f\|_{\mathcal{\dot{M}}^{p,l}(\Omega)}=\sup_{R>0}
\sup_{x\in \Omega}R^{3(\f1p-\f1l)}\B(\int_{B_{x}(R)\cap\Omega}|f(y)|^{l}dy\B)^{\f1l}<\infty.
$$
Here $B_{x}(R)$ represents the open ball centered at $x\in \mathbb{R}^{3}$ with radius $R>0$. In particular, by using the Lebesgue differentiation theorem, one can easily prove that $\mathcal{\dot{M}}^{\infty,l}(\Omega)=L^{\infty}(\Omega)$.

To give the definition of Besov spaces, we denote $P_{t}=e^{t\Delta}$ as the heat semigroup on $\mathbb{R}^{n}$. For $\alpha<0$, a tempered distribution $f$ on $\mathbb{R}^{n}$ belongs to the Besov space $\dot{B}^{\alpha}_{\infty,\infty}(\mathbb{R}^{n})$ if and only if the following norm
$$
\|f\|_{\dot{B}^{\alpha}_{\infty,\infty}(\mathbb{R}^{n})}=\sup_{t>0}
t^{-\f{\alpha}{2}}\|P_{t}f\|_{L^{\infty}(\mathbb{R}^{n})}
$$ is finite.

As usual, given a Schwartz function $f$ on $\mathbb{R}^{n}$, the Fourier transform $\hat{f}$ of $f$ is given by $$\hat{f}(\xi)=\frac{1}{(2\pi)^{n}}\int_{\mathbb{R}^{n}}f (x)e^{-i\xi\cdot x}\,dx, $$
and the inverse Fourier transform $f^{\vee}$ is defined as
$$
f^{\vee}(\xi)=\widehat{f}(-\xi)
$$
for all $\xi \in \mathbb{R}^{n}.$ Furthermore, for $s\geq 0,$ we define $\Lambda^{s}f $ by
$$\widehat{\Lambda^{s} f}(\xi)=|\xi|^{s}\hat{f}(\xi),$$
where the notation $\Lambda$ stands for the square root of the negative Laplacian $(-\Delta)^{1/2}$. Now introduce
the homogenous Sobolev norm $\|\cdot\|_{\dot{H}^{s}(\mathbb{R}^{n})}$ as follows $$\|f\|_{\dot{H}^{s}(\mathbb{R}^{n})}=
\|\Lambda^{s}f\|_{L^{2}(\mathbb{R}^{n})}.$$
Denote $ H^{s}$ as the standard inhomogenous Sobolev space with the norm $$\|f\| _{H^{s}(\mathbb{R}^{n})}= \|\Lambda^{s}f\|_{L^{2}(\mathbb{R}^{n})}+\|f\|_{L^{2}(\mathbb{R}^{n})}.
$$
For $q\in [1,\,\infty]$, the notation $L^{q}(0,\,T;\,X)$ stands for the set of all measurable functions $f(x,t)$ on the interval $(0,\,T)$ with values in $X$ and $\|f(\cdot,t)\|_{X}$ belonging to $L^{q}(0,\,T)$.
Throughout this paper, we denote
\begin{align*}
     & B(r):=\{y\in \mathbb{R}^{3}|\,|y|< r\}, && & Q(r):=B(r)\times (T-r^2,T).
      \end{align*}
The average integral of a function $h$ on the ball $B(r)$ is defined by
$\overline{h}_{r}:=\frac{1}{|B(r)|}\int_{B(r)}h\,$. For simplicity, we write
$$\ba
\|\cdot\| _{L^{p}L^{q}(Q(r))}:=&\|\cdot\| _{L^{p}(T-r^{2},T;L^{q}(B(r)))} ~~~ \text{and}~~~
 \|\cdot\| _{L^{p}(Q(r))}:=\|\cdot\| _{L^{p}L^{p}(Q(r))}.
 \ea$$
  We will use the summation convention on repeated indices.
 $C$ is an absolute constant which may be different from line to line unless otherwise stated in this paper.
\subsection{Suitable weak solutions}

\begin{definition}\label{defweak}
A divergence-free vector field $U = (U_1,U_2,U_3)\in W^{1,2}_{\rm loc}(\R^3)$  is called a weak solution of \eqref{SNS}, if for all divergence-free vector field $\phi = (\phi_1,\phi_2,\phi_3) \in C^\infty_0(\R^3)$ one has
\begin{equation}\label{weakform}
\int_{\R^3}(  \nabla U\cdot\nabla \phi +[aU+a(y\cdot \nabla)U+(U\cdot\nabla) U]\cdot \phi) dy   = 0.
\end{equation}
\end{definition}

Now,    for the convenience of readers, we recall the \wred{classical definition of   suitable weak solutions} to the Navier-Stokes system \eqref{NS}.

	\begin{definition}\label{defi1}
		A  pair   $(u, \, \pi)$  is called a suitable weak solution to the Navier-Stokes equations \eqref{NS} provided the following conditions are satisfied,
		\begin{enumerate}[(1)]
			\item $u \in L^{\infty}(-T,\,0;\,L^{2}(\mathbb{R}^{3}))\cap L^{2}(-T,\,0;\,\dot{H}^{1}(\mathbb{R}^{3})),\,\pi\in
			L^{3/2}(-T,\,0;L^{3/2}(\mathbb{R}^{3}));$
			\item$(u, ~\pi)$~solves (\ref{NS}) in $\mathbb{R}^{3}\times (-T,\,0) $ in the sense of distributions;
			\item$(u, ~\pi)$ satisfies the following inequality, for a.e. $t\in[-T,0]$,
			\begin{align}
				&\int_{\mathbb{R}^{3}} |u(x,t)|^{2} \phi(x,t) dx
				+2\int^{t}_{-T}\int_{\mathbb{R} ^{3 }}
				|\nabla u|^{2}\phi  dxds\nonumber\\ \leq&  \int^{t}_{-T }\int_{\mathbb{R}^{3}} |u|^{2}
				(\partial_{s}\phi+\Delta \phi)dxds
				+ \int^{t}_{-T }
				\int_{\mathbb{R}^{3}}u\cdot\nabla\phi (|u|^{2} +2\pi)dxds, \label{loc}
			\end{align}
			where non-negative function $\phi(x,s)\in C_{0}^{\infty}(\mathbb{R}^{3}\times (-T,0) )$.\label{SWS3}
		\end{enumerate}
	\end{definition}

\begin{lemma}\cite{[HWZ]}\label{lem2}
Let $\Phi$ denote  the standard normalized fundamental solution of Laplace equation in $\mathbb{R}^{3}$. For $0<\xi<\eta$, we consider smooth cut-off function $\psi\in C^{\infty}_{0}(B(\f{\xi+3\eta}{4}))$ such that $0\leq\psi\leq1$ in $B(\eta)$, $\psi\equiv1$ in $B(\f{3\xi+5\eta}{8})$  and $|\nabla^{k}\psi |\leq C/(\eta-\xi)^{k}$ with $k=1,2$ in  $B(\eta)$.  Then we may split  pressure $\pi$ in \eqref{NS} as below
			\be\label{decompose pk}
			\pi(x):=\pi_{1}(x)+\pi_{2}(x)+\pi_{3}(x), \quad x\in B(\f{\xi+\eta}{2}),
			\ee
			where
$$\ba			\pi_{1}(x)=&-\partial_{i}\partial_{j} \Phi  \ast (u_{j}u_{i} \psi) ,\\
			\pi_{2}(x)
			=&2\partial_{i} \Phi  \ast (u_{j}u_{i}\partial_{j}\psi)- \Phi  \ast
			(u_{j}u_{i} \partial_{i}\partial_{j}\psi ) , \\
			\pi_{3}(x)
			=&2\partial_{i} \Phi  \ast (\pi \partial_{i}\psi )- \Phi  \ast (\pi \partial_{i}\partial_{i}\psi ).
			\ea
			$$
 Moreover, there holds
			\begin{align}
				& \| \pi_1\|_{L^{3/2}(Q(\f{\xi+\eta}{2}))}
				\leq C\|  u\|^{2}_{L^{3}(Q(\f{\xi+3\eta}{4}))};\label{p1estimate}\\
				& \|  \pi_2\|_{L^{3/2}(Q(\f{\xi+\eta}{2}))}
				\leq  \f{C\eta^{3}}{(\eta-\xi)^{3}}\|  u\|^{2}_{L^{3}(Q(\f{\xi+3\eta}{4}))};\label{p2estimate}\\
				& \|  \pi_3\|_{L^{1}L^{2}(Q(\f{\xi+\eta}{2}))}
				\leq \f{C\eta^{3/2 }}{(\eta-\xi)^{3}}\|\pi\|_{L^{1}(Q(\f{\xi+3\eta}{4}))}.\label{p3estimate}
			\end{align}
		\end{lemma}
\subsection{Local suitable weak solutions}
We begin with \wgr{the Wolf's} local pressure projection $\mathcal{W}_{p,\Omega}:$ $W^{-1,p}(\Omega)\rightarrow W^{-1,p}(\Omega)$ $(1<p<\infty)$.
 More precisely, for any  $f\in W^{-1,p}(\Omega)$, we define \wgr{$\mathcal{W}_{p,\Omega}(f)= \nabla\pi$}, where $\pi$ satisfies \eqref{GMS}.
Let $\Omega$  be a  bounded domain with $\partial\Omega\in C^{1}$.
According to the $L^p$ theorem of Stokes system in \cite[Theorem 2.1, p.149]{[GSS]},
there exists a unique pair $(u,\pi)\in W^{1,p}(\Omega)\times L^{p}(\Omega)$ such that
\be\label{GMS}
-\Delta u+\nabla\pi=f,~~ \text{div}\,u=0, ~~u|_{\partial\Omega}=0,~~ \int_{\Omega}\pi dx=0.
\ee
Moreover, this pair is subject to the inequality
$$
\|u\|_{\wgr{W^{1,p}}(\Omega)}+\|\pi\|_{\wgr{L^p}(\Omega)}\leq C\|f\|_{\wgr{W^{-1,p}}(\Omega)}.
$$
Let $\nabla\pi= \mathcal{W}_{p,\Omega}(f)$ $(f\in L^p(\Omega))$, then $\|  \pi\|_{L^p(\Omega)}\leq C\|f\|_{L^p(\Omega)},$ where we used the fact that $L^{p}(\Omega)\hookrightarrow W^{-1,p}(\Omega)$.  Moreover, from $\Delta \pi=\text{div}\,f$, we see that $\|  \nabla\pi\|_{L^p(\Omega)}\leq C(\|f\|_{L^p(\Omega)}+ \|  \pi\|_{L^p(\Omega)}) \leq C\|f\|_{L^p(\Omega)}.$
Now, we  present the definition of local suitable weak solutions to Navier-Stokes equations \eqref{NS}.
	\begin{definition}\label{defilsw}
		A  pair   $(u, \,\pi)$  is called a local suitable weak solution to the Navier-Stokes equations \eqref{NS} provided the following conditions are satisfied,
		\begin{enumerate}[(1)]
			\item $u \in L^{\infty}(-T,\,0;\,L^{2}(\mathbb{R}^{3}))\cap L^{2}(-T,\,0;\,\dot{H}^{1}(\mathbb{R}^{3})),\,\pi\in
			L^{3/2}(-T,\,0;L^{3/2}(\mathbb{R}^{3}));$\label{SWS1}
			\item$(u, ~\pi)$~solves (\ref{NS}) in $\mathbb{R}^{3}\times (-T,\,0) $ in the sense of distributions;\label{SWS2}
			\item The local energy inequality reads,
for a.e. $t\in[-T,0]$ and non-negative function $\phi(x,s)\in C_{0}^{\infty}(\mathbb{R}^{3}\times (-T,0) )$,
			 \begin{align}
  &\int_{B(R)}|v|^2\phi (x,t)  d  x+ \int^{t}_{-T }\int_{B(R)}\big|\nabla v\big|^2\wgr{\phi (x,s) dx ds}\nonumber\\  \leq&   \int^{t}_{-T }\int_{B(R)} | v |^2(  \Delta \phi +  \partial_{s}\phi )  d  x d s +\int^{t}_{-T }\int_{B(R)}|v|^{2}u\cdot\nabla \phi    dxds\nonumber\\
& +\int^{t}_{-T }\int_{B(R)} \phi ( u\otimes v : \nabla^{2}\pi_{h} )dxds   +\int^{t}_{-T }\int_{B(R)} \phi \pi_{1}v\cdot\nabla \phi   dxds+\int^{t}_{-T }\int_{B(R)} \phi \pi_{2}v\cdot\nabla \phi   dxds.\label{wloc1}
 \end{align}
\wgr{Here, $\nabla\pi_h=-\mathcal{W}_{p,B(R)}(u)$, $\nabla\pi_1=\mathcal{W}_{p,B(R)}(\Delta u)$, $\nabla\pi_2=-\mathcal{W}_{p,B(R)}(u\cdot\nabla u)$, $v=u+\nabla\pi_h$.}
 In addition, $ \nabla\pi_{h}, \nabla\pi_{1}$ and $\nabla\pi_{2}$ meet the following \wgr{facts}
	\begin{align}   &\|\nabla\pi_{h}\|_{L^p(B(R))}\leq  \wgr{C}\|u\|_{L^p(B(R))}, \label{ph}\\
 &\|\wgr{\pi_{1}}\|_{L^2(B(R))}\leq \wgr{C} \|\nabla u\|_{L^2(B(R))},\label{p1}\\
 &\|\wgr{\pi_{2}}\|_{L^{p/2}(B(R))}\leq \wgr{C} \| |u|^{2}\|_{L^{p/2}(B(R))}.\label{p2}
\end{align}	\end{enumerate}
	\end{definition}

We list some
interior estimates
of the harmonic equation $\Delta h=0$, which will be frequently utilized later. Let $1\leq p,q\leq\infty$ and $0<r<\rho$, then there holds
\be\label{h1}\|\nabla^{k}h\|_{L^{q}
(B(r))}\leq \f{Cr^{\f{3}{q}}}{(\rho-r)^{\f{3}{p}+k}}\|h\|_{L^{p}(B(\rho))}\wgr{,}\ee
\be\label{h2}
 \| h-\overline{h}_{r}\|_{L^{q}
(B(r))}\leq \f{Cr^{\f{3}{q}+1}}{(\rho-r)^{\f{3}{q}+1 }}\|h-\overline{h}_{\rho}\|_{L^{q}(B(\rho))}.\ee
 The proof of \eqref{h1} rests on the mean value property of harmonic functions. This together with
  mean value theorem leads to \eqref{h2}. We leave the details  to the reader.

\subsection{Meyer-Gerard-Oru  interpolation inequality}
For the convenience of the readers, we recall
Meyer-Gerard-Oru  interpolation inequality in \cite{[MGO],[PP]} below
\be\label{mgoi}
\|f\|_{L^{m}(\mathbb{R}^{n})}\leq C \|f\|^{\f2m}_{\dot{H}^{s}(\mathbb{R}^{n})}
 \|f\|^{1-\f2m}_{\dot{B}^{-\f{2s}{m-2}}_{\infty,\infty}(\mathbb{R}^{n})} ~~\text{with} ~~s>0~~\text{and}~~2<m<\infty.\ee
In addition,we have the following embedding relation between Morrey spaces and Besov spaces (see \cite[Section 3.2, Lemma 1]{[PP]})
\be\label{BM}
\|f\|_{\dot{B}^{-\f{3}{q}}_{\infty,\infty}(\mathbb{R}^{3})}\leq C\|f
\|_{\mathcal{\dot{M}}^{q,1}(\mathbb{R}^{3})} ~~\text{with} ~~1<q<\infty.
\ee
These two inequalities play  an important role in our proof.

\section{Proof of Theorem \ref{the1.1} }
In this section, we present the proof of Theorem \ref{the1.1}. In Subsection 3.1, we define $\Pi$ by $\mathcal{R}_i \mathcal{R}_j(U_{i}U_{j})$ and examine that it satisfies the equations \eqref{SNS}. In the second subsection we establish the energy bounds. Invoking energy bounds  and \eqref{local-energy}, we get the desired results in the last subsection.
\subsection{The construction of the pressure}
\begin{lemma} \label{cp}
Suppose that  $U$ is a weak solution of \eqref{SNS} and $U\in \mathcal{\dot{M}}^{2p,2q}(\mathbb{R}^{3})$ with $1<q\leq p <\infty$.
Let $\widetilde{\Pi}=\mathcal{R}_i \mathcal{R}_j(U_{i}U_{j})$. Then there holds
\begin{enumerate}[(1)]
			\item
  \be\label{pm}
\|\widetilde{\Pi}\|_{\mathcal{\dot{M}}^{p,q}(\mathbb{R}^{3})}\leq C \|U\|^{2}_{\mathcal{\dot{M}}^{2p,2q}(\mathbb{R}^{3})}.
\ee
 \item  $\widetilde{\Pi}$ meets
$$-\Delta \Pi=\partial_{i}\partial_{j}(U_{i}U_{j}),$$
in the distributional sense.

\item  $(U,\widetilde{\Pi})$ smoothly solves
\begin{equation}\label{prescon}
-\nu \Delta U+aU+a(y\cdot\nabla)U+(U\cdot\nabla)U+\nabla \Pi=0 \quad {\rm in} \,\,\, \mathbb{R}^3.
\end{equation}
Here the solution $\Pi$ is unique up to a constant.
\end{enumerate}
\end{lemma}

\begin{proof}
(1) The boundedness of
 singular integral operators $R_iR_j$  on Morrey spaces $\mathcal{\dot{M}}^{p,q}(\mathbb{R}^{3})$ with $1<q\leq p <\infty$ means \eqref{pm}.

 (2)
To show the second assertion, assume for a while, we have proved that
 \be\ba\label{cut}
\int_{\mathbb{R}^3} \mathcal{R}_i\mathcal{R}_j( \chi_{B(R)} U_{i}U_{j})  \Delta\varphi dx
=&- \int_{\mathbb{R}^3}\chi_{B(R)} U_iU_j  \partial_i\partial_j \varphi dx, ~~~\varphi\in C^{\infty}_{0}(\mathbb{R}^3),
\ea\ee
where  $\chi_{B(R)}$ is  the characteristic function of $B(R)$. The problem reduces to that of passing to the limit as  $R$ tends to infinity. Indeed, note that, for any $f\in \mathcal{\dot{M}}^{p,q}(\mathbb{R}^{3})$ and $b>3-\f{3q}{p}$,
$$
\int_{\mathbb{R}^{3}}|f|^{q}(1+|x|^{2})^{-\f{b}{2}}dx \leq
C \|f\|^{q}_{\mathcal{\dot{M}}^{p,q}(\mathbb{R}^{3})}.
$$
This fact can be found in \cite{[Kato]}.
To make our paper more self-contained and more readable, we present the proof of the above fact as follows
$$\ba
\int_{\mathbb{R}^{3}}|f|^{q}(1+|x|^{2})^{-\f{b}{2}}dx &\leq
  \int_{|x|\leq1}|f|^{q}(1+|x|^{2})^{-\f{b}{2}}dx
  +\int_{|x|>1}|f|^{q}(1+|x|^{2})^{-\f{b}{2}}dx.
\\&\leq    \int_{|x|\leq1}|f|^{q} dx+\sum_{k=0 }^{\infty}\int_{2^{k}<|x|\leq2^{k+1}}|f|^{q}
(1+|x|^{2})^{-\f{b}{2}}dx\\
&\leq  \|f\|^{q}_{\mathcal{\dot{M}}^{p,q}(\mathbb{R}^{3})}+C\sum_{k=0 }^{\infty}\big(2^{k+1}\big)^{-b-3(\f{q}{p}-1)} \|f\|^{q}_{\mathcal{\dot{M}}^{p,q}(\mathbb{R}^{3})}\\
&\leq  C\|f\|^{q}_{\mathcal{\dot{M}}^{p,q}(\mathbb{R}^{3})}.
\ea$$
In addition, note that $w(x)=(1+|x|^{2})^{-\f{b}{2}}$ satisfies the Muckenhoupt $A_{q}$-condition if $0\leq b<3$. Therefore, for $p<\infty$, we choose $b$ such that $3-\f{3 q}{p}<b<3$. Then, the H\"older inequality and the classical Calder\'on-Zygmund Theorem with $A_{q}$ weights yield that
$$\ba
&\left|\int_{\mathbb{R}^3}\B(\mathcal{R}_i\mathcal{R}_j(U_iU_j )-\mathcal{R}_i\mathcal{R}_j( \chi_{B(R)} U_iU_j )\B)  \Delta\varphi dx\right| \\
&\leq  \left(\int_{\mathbb{R}^3}\left|\mathcal{R}_i\mathcal{R}_j(U_iU_j- \chi_{B(R)} U_iU_j )\right|^q w(x) dx\right)^{\frac1q} \left(\int_{\mathbb{R}^3} |  \Delta\varphi|^{\f{q}{q-1}} w(x)^{\f{1}{1-q}} dx\right)^{1-\frac1q}\\
&\leq C \left(\int_{\mathbb{R}^3}\left| U_iU_j- \chi_{B(R)} U_iU_j  \right|^q w(x) dx\right)^{\frac1q} \left(\int_{\mathbb{R}^3} |  \Delta\varphi|^{\f{q}{q-1}} w(x)^{\f{1}{1-q}} dx\right)^{1-\frac1q}\\
&\leq C \|U_iU_j\|_{\mathcal{\dot{M}}^{p,q}(\mathbb{R}^{3})}
,
\ea$$
Likewise,
$$\ba
&\left|\int_{\mathbb{R}^3}\B(  U_iU_j  - \chi_{B(R)} U_iU_j \B)   \partial_{i}\partial_{j}\varphi dx\right| \\
&\leq C \|U_iU_j\|_{\mathcal{\dot{M}}^{p,q}(\mathbb{R}^{3})}.
\ea$$
Thus we can apply
the Lebesgue's dominated convergence theorem to obtain
$$
\lim_{R\rightarrow\infty}\int_{\mathbb{R}^3}\mathcal{R}_i\mathcal{R}_j( \chi_{B(R)} U_iU_j ) \Delta\varphi dx=\int_{\mathbb{R}^3}\mathcal{R}_i\mathcal{R}_j(U_iU_j )  \Delta\varphi dx,
$$
and
\be\ba\lim_{R\rightarrow\infty}
 \int_{\mathbb{R}^3} \chi_{B(R)} U_{i}U_{j}   \partial_i\partial_j\varphi dx = \int_{\mathbb{R}^3} U_iU_j  \partial_i\partial_j \varphi dx. \ea\ee
Combining this and \eqref{cut}, we know that
  \be\ba\label{cut2}
\int_{\mathbb{R}^3} \mathcal{R}_i\mathcal{R}_j(  U_{i}U_{j})  \Delta\varphi dx
=&- \int_{\mathbb{R}^3} U_iU_j  \partial_i\partial_j \varphi dx,
\ea\ee
which implies the second assertion. Subsequently, we need to prove \eqref{cut} we have assumed.
Since the Fourier transform is a topological isomorphism from
 $\mathcal{S}(\mathbb{R}^{n})$  onto itself, we conclude that $(\Delta\varphi)^{\vee}\in\mathcal{S}(\mathbb{R}^{3})$
 from $\Delta\varphi \in C_{0}^{\infty}(\mathbb{R}^{3})\subset\mathcal{S}(\mathbb{R}^{3})$.
Moreover, there holds $\Delta\varphi= ((\Delta\varphi)^{\vee})^{\wedge}$. Let $s=\min\{q,2\}$.
 According to the definition of the Fourier transform of  tempered distributions and
  $\mathcal{R}_i\mathcal{R}_j( \chi_{B(R)} U_{i}U_{j})\in L^{s}(\mathbb{R}^{3})\subset \mathcal{S}'(\mathbb{R}^{3})$, we discover
   \be\ba\label{cut1}
\int_{\mathbb{R}^3} \mathcal{R}_i\mathcal{R}_j( \chi_{B(R)} U_{i}U_{j})  ((\Delta\varphi)^{\vee})^{\wedge} dx=
\int_{\mathbb{R}^3} (\mathcal{R}_i\mathcal{R}_j( \chi_{B(R)} U_{i}U_{j})) ^{\wedge} (\Delta\varphi)^{\vee} dx.
\ea\ee
We recall the following fact (see \cite[p.76]{[Duoandikoetxea]}) that
 for any $f\in L^{m}(\mathbb{R}^{3})$ with $1<m\leq 2$,
\be\label{xfq}
(\mathcal{R}_i\mathcal{R}_jf)^{\wedge}(\xi)=
 -\f{\xi_{i}\xi_{j}}{|\xi|^{2}}\hat{f}(\xi), ~~~ \xi\in\mathbb{R}^{3}.
 \ee
From
\eqref{cut1} and
\eqref{xfq}, we observe that
 \be\ba\label{cut5}
\int_{\mathbb{R}^3} \mathcal{R}_i\mathcal{R}_j( \chi_{B(R)} U_{i}U_{j})  \Delta\varphi dx
=& \int_{\mathbb{R}^3}( \mathcal{R}_i\mathcal{R}_j( \chi_{B(R)} U_{i}U_{j})) ^{\wedge} (\Delta\varphi)^{\vee} dx
\\
=&- \int_{\mathbb{R}^3} \f{\xi_{i}\xi_{j}}{|\xi|^{2}}(  \chi_{B(R)} U_{i}U_{j} ) ^{\wedge}( |\xi|^{2}\varphi^{\vee}) dx
\\
=&- \int_{\mathbb{R}^3}  \chi_{B(R)} U_{i}U_{j} (\xi_{i}\xi_{j}\,\varphi^{\vee} ) ^{\wedge} dx\\
=&- \int_{\mathbb{R}^3}\chi_{B(R)} U_iU_j  \partial_i\partial_j \varphi dx.
\ea\ee
This confirms \eqref{cut}.

Next, we turn our attention to demonstrating  \eqref{prescon}. Before going further, we write
\begin{equation*}%\label{defF}
F:= -\nu \Delta U+aU+a(y\cdot\nabla)U+(U\cdot\nabla)U+\nabla \widetilde{\Pi}
\end{equation*}
 and show that $F\equiv 0.$

It follows from  $F=\nabla \widetilde{\Pi}-\nabla \Pi$, div $F=0$ and curl $F=0$ that $\Delta F=0$. Since harmonic functions are analytic, to get $F\equiv0$, it suffices to show that $D^{\alpha}F(0)=0$
for any multi-index $\alpha=(\alpha_1,\alpha_2,\alpha_3)$ with $|\alpha|\geq 0$.
Let $\theta\in C^{\infty}_{0}(\mathbb{R}^{3})$ be any radial function which is supported in $\{x\in \mathbb{R}^{3}|\,|x|<1\}$ and has integral 1. Note that $D^{\alpha}F$ is also harmonic, there holds
\begin{equation}\label{MVP}
D^{\alpha}F(0)=\varepsilon^3\int_{\mathbb{R}^3} D^{\alpha}F(y)\,\theta(\varepsilon y)dy
\end{equation}
for any $\varepsilon>0$. (See \cite[p.275]{[Stein1]} for the details.)

  Integrating by parts, one computes
\begin{equation*}%\label{MVPD}
D^\alpha F(0)=\varepsilon^3\int_{\mathbb{R}^3} D^\alpha F(y)\,\theta(\varepsilon y)dy =(-1)^{|\alpha|}\varepsilon^3\int_{\mathbb{R}^3}  F(y) \varepsilon^{|\alpha|} (D^\alpha\theta)(\varepsilon y)dy.
\end{equation*}
According to this, to show that $D^\alpha F(0)=0$, it suffices to prove that for any $\varphi\in C^{\infty}_{0}(B(1))$,
$$\lim_{\varepsilon\to 0^+}\varepsilon^3\int_{\mathbb{R}^3}F(y)\varphi(\varepsilon y)dy=0.
$$
Integration by parts twice, the H\"older inequality and the definition of Morrey spaces guarantee that
$$\ba
\left|\varepsilon^{3}\int_{\mathbb{R}^{3}}\Delta U\varphi(\varepsilon y)dy\right|=&\left|\varepsilon^{5}\int_{\mathbb{R}^{3}} U\Delta\varphi(\varepsilon y)dy\right|\\
\leq& \varepsilon^{5} \B(\int_{|y|<\f{1}{\varepsilon}}|U|^{2q}dy\B)^{\f{1}{2q}}
\B(\int_{|y|<\f{1}{\varepsilon}}|\Delta\varphi(\varepsilon y)|^{\f{2q}{2q-1}}dy\B)^{1-\f{1}{2q}}\\
\leq& C \varepsilon^{2+\f{3}{2p}}
\|U\|_{\mathcal{\dot{M}}^{2p,2q}(\mathbb{R}^{3})}.
\ea$$
A slight modification of the latter argument yields
$$\ba
\left|\varepsilon^{3}\int_{\mathbb{R}^{3}}U\varphi(\varepsilon y)dy\right| \leq& \,\varepsilon^{3} \B(\int_{|y|<\f{1}{\varepsilon}}|U|^{2q}dy\B)^{\f{1}{2q}}
\B(\int_{|y|<\f{1}{\varepsilon}}|\varphi(\varepsilon y)|^{\f{2q}{2q-1}}dy\B)^{1-\f{1}{2q}}\\
\leq& C \varepsilon^{\f{3}{2p}}
\|U\|_{\mathcal{\dot{M}}^{2p,2q}(\mathbb{R}^{3})}.
\ea$$
In virtue of  integration
by parts once again, we see that
\be\ba
\varepsilon^3\int_{\mathbb{R}^{3}}(y\cdot\nabla) U(y)\varphi(\varepsilon y) dy = -\varepsilon^3\int_{\mathbb{R}^{3}} 3U(y)\varphi(\varepsilon y)dy- \varepsilon^3 \int_{\mathbb{R}^{3}} U(y)\{(\varepsilon y)\cdot \nabla\varphi(\varepsilon y)\}dy.
\ea\ee
In the same manner as above, we can also bound $\left|\varepsilon^3\int_{\mathbb{R}^{3}}(y\cdot\nabla) U(y)\varphi(\varepsilon y) dy\right|$.

On the other hand, it follows from the divergence-free condition that
\be\ba
&\varepsilon^3\int_{\mathbb{R}^{3}}(U\cdot\nabla) U(y)\varphi(\varepsilon y) dy\\
 =& \varepsilon^3\int_{\mathbb{R}^{3}} \text{div}\,U\otimes U(y)\varphi(\varepsilon y)dy\\
=& -\varepsilon^{4}\int_{\mathbb{R}^{3}}  U\otimes U(y)\cdot\nabla\varphi(\varepsilon y)dy,
\ea\ee
which in turn implies that
\be\label{UUe}\ba
&\left|\varepsilon^3\int_{\mathbb{R}^{3}}(U\cdot\nabla) U(y)\varphi(\varepsilon y) dy\right|\\
\leq& \varepsilon^{4} \B(\int_{|y|<\f{1}{\varepsilon}}|U|^{2q}dy\B)^{\f1q}
\B(\int_{|y|<\f{1}{\varepsilon}}|\nabla\varphi(\varepsilon y)|^{\f{q}{q-1}}dy\B)^{1-\f1q}\\
\leq& C \varepsilon^{1+\f{3}{p}}
\|U\|^{2}_{\mathcal{\dot{M}}^{2p,2q}(\mathbb{R}^{3})}.
\ea\ee
Eventually, we need to bound the last term involving $\nabla\widetilde{\Pi}$. Integration by parts gives
\be\ba
\varepsilon^3\int_{\mathbb{R}^{3}}\nabla \widetilde{\Pi}(y)\,\varphi(\varepsilon y) dy
  =-\varepsilon^{4}\int_{\mathbb{R}^{3}}\widetilde{\Pi}(y)\,\nabla\varphi(\varepsilon y)dy.
\ea\ee
Furthermore, a variant of (\ref{UUe}) provides the estimate
$$
\left|\varepsilon^{4}\int_{\mathbb{R}^{3}} \widetilde{\Pi}(y)\,\nabla\varphi(\varepsilon y)dy\right|
\leq C \varepsilon^{1+\f{3}{p}}\|\widetilde{\Pi}\|_{\mathcal{\dot{M}}^{p,q}(\mathbb{R}^{3})} \leq C \varepsilon^{1+\f{3}{p}}\|U\|^{2}_{\mathcal{\dot{M}}^{2p,2q}(\mathbb{R}^{3})}.
$$
This verifies $D^\alpha F(0)=0$ and completes the proof of Lemma \ref{cp}.
\end{proof}
\subsection{Energy bounds }\label{sec4}

\begin{prop}\label{the1.2}Assume that  $(u,  \pi)$ is a suitable weak solution to the 3D Navier-Stokes system \eqref{NS}. Let $\alpha=2/p+3/q$ and the pair  $(p,q)$ satisfy
 \be\label{pqcondi}
1\leq2/p+3/q<2, ~~~\text{with}~~~ 1<p<\infty.
 \ee
 Then there holds, for any $m\geq1$,
\be\ba \label{key ineq}
&\|u\|^2_{L^{\infty}(T-\f{R^{2}}{4},T;L^{2}(B(R/2)))}+\|\nabla u\|^2_{L^{2}(T-\f{R^{2}}{4},T;L^{2}(B(R/2)))} \\ \leq&    CR^{3-2\alpha} \|u\|_{L^{p}(T-R^{2},T;\dot{\mathcal{M}}^{q,1} (B(R)))} ^{2}\\& +  CR^{\f{6-5\alpha}{2-\alpha}} \|u\|_{L^{p}(T-R^{2},T;\dot{\mathcal{M}}^{q,1} (B(R)))}  ^{\f{4}{2-\alpha}} + CR^{1-\f{6}{m}} \|\pi\|^{2}_{L^{1}(T-R^{2},T;\dot{\mathcal{M}}^{m,1} (B(R)))}.
\ea\ee
\end{prop}
\begin{remark}
We refer the readers to \cite{[GP]} and \cite{[HWZ]} for recent progress on energy bounds of suitable weak solutions.
\end{remark}
In the spirit of \cite{[GP],[HWZ]}, we conclude
$\varepsilon$-regularity criteria at one scale in Morrey spaces from the above energy bounds.  Previous  related results in  Lorentz spaces  can be found in \cite{[WWY],[Baker]}. In addition, a summary of
$\varepsilon$-regularity criteria at one scale is given in \cite{[HWZ]}.
\begin{coro}\label{the1.3}
		Let  the pair $(u,  \pi)$ be a suitable weak solution to the 3D Navier-Stokes system \eqref{NS} in $Q(1)$.
		There exists an absolute positive constant $\varepsilon_{1}$
		such that if the pair $(u,\pi)$ satisfy	 \be\label{0il1}\ba
&\|u\|_{ L^{p }(-1,0;\mathcal{\dot{M}}^{q,1}(B(1))) }+\|\pi\|_{ L^{1}(-1,0;\mathcal{\dot{M}}^{m,1}(B(1))) }\leq \varepsilon_{1},\ea\ee
where
$$1\leq 2/p+3/q <2, 1< p<\infty ~\text{and}~ m\geq1,$$
	 then  $u\in L^{\infty}(Q(1/2)).$
	\end{coro}
\begin{remark}
This result extends regularity criteria via Lebesgue spaces in \cite{[HWZ]} to Morrey spaces.
We refer the readers to \cite{[GKT]} for various $\varepsilon$-regularity criteria   at all scales in Lebesgue spaces.
\end{remark}
For abbreviation, we set
$$\ba
 E(r)=\sup_{T-r^{2}\leq t\leq T}\|u(\cdot,t)\|^{2}_{L^{2}(B(r))}+\int_{T-r^{2}}^{T}\|\nabla u\|^{2}_{L^{2}(B(r))}dt .
\ea$$
\begin{lemma}\label{ineq}
For $0<\xi<\eta$, let $r=\f{\xi+3\eta}{4}$ and $\alpha=2/p+3/q$ with the pair $(p,q)$ satisfying
\eqref{pqcondi}. Then
there  exits  an absolute constant $C$  independent of  $\xi$ and $\eta$,~ such that
\be\ba\label{keyineq}
 \int_{T-r^{2}}^{T}\|u\|^{3}_{L^{3}(B(r))}dt
 \leq  C \eta^{\f{3(2-\alpha)}{2\alpha} } \B(1+\f{\eta^{2}}{(\eta-\xi)^{2}}\B)^{\f32-\f{1}{\alpha}}E^{\f32-\f{1}{\alpha}}(\eta)
\B(\int_{T-r^{2}}^{T}\|u\|^{p}_{\dot{\mathcal{M}}^{q,1}(B(\eta))}dt\B)^{\f{2}{p\alpha}}.
\ea\ee
\end{lemma}
\begin{proof}
For any  pair $(p,q)$ satisfying \eqref{pqcondi}, we can
select $p'<p$ such that the value of $\f{2}{p'}+\f3q$ is very close to $2$ and smaller than $2$. Due to the H\"older inequality only in time direction, it is enough to consider the case that $\f2p+\f3q$ is close to 2.
To proceed further, we set
$$m=\f{6\alpha}{3\alpha-2}.$$
Taking $\alpha\rightarrow 2^{-}$, we conclude from some elementary computations that
\be
\label{mxianzhi}
3<m\leq \f{2q+6}{3}.\ee
 Invoking interpolation inequality \eqref{mgoi}, we see that
\be\label{2.9}\ba
\|u\|_{L^{m}(\mathbb{R}^{3})}\leq& C \|u\|^{\f2m}_{\dot{H}^{\f{3(m-2)}{2q}}(\mathbb{R}^{3})}\|u\|^{1-\f2m}
_{\dot{B}_{\infty,\infty}^{-\f3q}(\mathbb{R}^{3})}
\\\leq& C \|u\|^{\f2m}_{\dot{H}^{\f{3(m-2)}{2q}}(\mathbb{R}^{3})}\|u\|^{1-\f2m}
_{\mathcal{\dot{M}}^{q,1}(\mathbb{R}^{3})},
\ea\ee
where we also used \eqref{BM}.

Using the H\"older inequality and \eqref{2.9}, we arrive at
\be\label{2.11}\ba
\|u\|^{3}_{L^{3}(B(r))}\leq& C r^{9(\f13-\f{1}{m})}\|u\|^{3}_{L^{m}(B(r))}\\\leq& C r^{9(\f13-\f{1}{m})}\|u\|^{3}_{L^{m}(\mathbb{R}^{3})}\\
\leq& C r^{9(\f13-\f{1}{m})}\|u\|^{\f6m}_{\dot{H}^{\f{3(m-2)}{2q}}(\mathbb{R}^{3})}
\|u\|^{3-\f6m}_{\mathcal{\dot{M}}^{q,1}(\mathbb{R}^{3})}.
\ea\ee
From the Gagliardo-Nirenberg inequality and \eqref{mxianzhi}, we get
$$
\|u\|_{\dot{H}^{\f{3(m-2)}{2q}}(\mathbb{R}^{3})}\leq C\|  u\|_{L^{2}(\mathbb{R}^{3})}^{\f{2q-3m+6}{2 q}}\|\nabla u\|^{\f{3m-6}{2q}}_{L^{2}(\mathbb{R}^{3})}.
$$
Substituting this into \eqref{2.11}, we deduce that
$$\ba
\|u\|^{3}_{L^{3}(B(r))}
\leq& C r^{9(\f13-\f{1}{m})}\| u\|_{L^{2}(\mathbb{R}^{3})}^{\f{3(2q-3m+6)}{m q}}\|\nabla u\|^{\f{3(3m-6)}{mq}}_{L^{2}(\mathbb{R}^{3})}
\|u\|^{3-\f6m}_{\mathcal{\dot{M}}^{q,1}(\mathbb{R}^{3})}.
\ea$$
Integrating this inequality in time and applying the H\"older inequality, we know that
$$\ba
&\int_{T-r^{2}}^{T}\|u\|^{3}_{L^{3}(B(r))}dt\\
\leq& C r^{9(\f13-\f{1}{m})} \sup_{  T-r^{2}\leq t\leq T}\| u\|_{L^{2}(\mathbb{R}^{3})}^{\f{3(2q-3m+6)}{m q}}\int^{T}_{T-r^{2}}\|\nabla u\|^{\f{3(3m-6)}{mq}}_{L^{2}(\mathbb{R}^{3})}
\|u\|^{3-\f6m}_{\mathcal{\dot{M}}^{q,1}(\mathbb{R}^{3})}dt\\
\leq& C r^{9(\f13-\f{1}{m})} \sup_{  T-r^{2}\leq t\leq T}\| u\|_{L^{2}(\mathbb{R}^{3})}^{\f{3(2q-3m+6)}{m q}}\B(\int_{T-r^{2}}^{T}\|\nabla u\|^{2}_{L^{2}(\mathbb{R}^{3})}dt\B)^{\f{3(3m-6)}{2mq}}\\&\times
\B(\int_{T-r^{2}}^{T}\|u\|^{\f{2q(3m-6)}{2qm-3(3m-6)}}
_{\mathcal{\dot{M}}^{q,1}(\mathbb{R}^{3})}dt\B)^{\f{2qm-3(3m-6)}{2qm}},
\ea$$
namely,
\be\label{4.5}\ba
&\int_{T-r^{2}}^{T}\|u\|^{3}_{L^{3}(B(r))}dt\\
\leq& C r^{\f{3(2-\alpha)}{2\alpha} } \B(\sup_{T-r^{2}\leq t\leq T}\| u\|^{2}_{L^{2}(\mathbb{R}^{3})}+\int_{T-r^{2}}^{T}\|\nabla u\|^{2}_{L^{2}(\mathbb{R}^{3})}dt\B)^{\f32-\f{1}{\alpha}}
\B(\int_{T-r^{2}}^{T}\|u\|^{p}_{\mathcal{\dot{M}}^{q,1}
(\mathbb{R}^{3})}dt\B)^{\f{2}{p\alpha}}.
\ea\ee
Choose a cut-off function $\psi\in C_{0}^{\infty}(B(\eta))$ satisfying $0\leq\psi\leq1$, $\psi\equiv1$ in $B(\f{\xi+3\eta}{4})$ and $|\nabla\psi|\leq \f{C}{|\eta-\xi|}$. Replacing $u$ by $\psi u$ in \eqref{4.5} and taking $r=\f{\xi+3\eta}{4}$, we get
$$\ba
&\int_{T-r^{2}}^{T}\|u\|^{3}_{L^{3}(B(r))}dt
\\\leq& C \eta^{\f{3(2-\alpha)}{2\alpha} } \B(1+\f{\eta^{2}}{(\eta-\xi)^{2}}\B)^{\f32-\f{1}{\alpha}}E^{\f32-\f{1}{\alpha}}(\eta)
\B(\int_{T-r^{2}}^{T}\|u\|^{p}_{\dot{\mathcal{M}}^{q,1}(B(\eta))}dt\B)^{\f{2}{p\alpha}}.
\ea$$
Here we used the fact below
$$
\|\psi u\|_{\dot{\mathcal{M}}^{q,1}(\mathbb{R}^{3})}\leq C
\|u\|_{\dot{\mathcal{M}}^{q,1}(B(\eta))},
$$
which can be derived from the definition of Morrey spaces. This completes the proof.
\end{proof}

\begin{proof}[Proof of  Proposition \ref{the1.2}]
Indeed, consider $0<R/2\leq \xi<\f{3\xi+\eta}{4}<\f{\xi+\eta}{2}<\f{\xi+3\eta}{4}<\eta\leq R$. Let $\phi(x,t)$ be a non-negative smooth function supported in $Q(\f{\xi+\eta}{2})$ such that
$\phi(x,t)\equiv1$ on $Q(\f{3\xi+\eta}{4})$,
$|\nabla \phi| \leq  C/(\eta-\xi) $ and $
|\nabla^{2}\phi|+|\partial_{t}\phi|\leq  C/(\eta-\xi)^{2} .$

The local energy inequality \eqref{loc}, the decomposition of  pressure  in Lemma \ref{lem2} and the H\"older inequality ensure that
\begin{align}
 &\int_{B(\f{\eta+\xi}{2})} |u(x,t)|^{2} \phi(x,t) dx
 +2\iint_{Q(\f{\eta+\xi}{2})}
  |\nabla u|^{2}\phi  dxds\nonumber\\\leq &  \f{C\eta^{5/3}}{(\eta-\xi)^{2}}\B(\iint_{Q(\f{\xi+3\eta}{4})} |u|^{3}
dxds\Big)^{2/3}+\f{C}{(\eta-\xi)} \iint_{Q(\f{\xi+3\eta}{4})}  |u|^{3} dxds\nonumber\\&+ \f{C\eta^{3}}{(\eta-\xi)^{4}}  \|u\|^{3}_{L^{3}(Q(\f{\xi+3\eta}{4}))}+\f{C\eta^{3/2}}{(\eta-\xi)^{4}} \|  \pi\|_{L^{1} (Q(\eta))}  \|u\|_{L^{\infty}L^2(Q(\eta)))}\nonumber\\
=&:I+II+III+IV.\label{last3}
 \end{align}
Combining
  \eqref{keyineq} and the Young inequality, we obtain
\begin{align}\nonumber
 I\leq &  \f{C\eta^{3+\alpha}}{(\eta-\xi)^{3\alpha}} \B(1+\f{\eta^{2}}{(\eta-\xi)^{2}}\B)^{ \f{3\alpha-2}{2}}
\|u\|_{L^{p}(T-\eta^{2},T;\dot{\mathcal{M}}^{q,1}(B(\eta)))}^{2}
+\f{1}{6}E(\eta),\\
\nonumber II
\leq & \f{C\eta^{3}}{(\eta-\xi)^{\f{2\alpha}{2-\alpha}}}   \B(1+\f{\eta^{2}}{(\eta-\xi)^{2}}\B)^{ \f{3\alpha-2}{2-\alpha}}
\|u\|_{L^{p}(T-\eta^{2},T;\dot{\mathcal{M}}^{q,1}(B(\eta)))}^{\f{4}{2-\alpha}}+\f{1}{6}E(\eta),\\ III\leq& \f{C\eta^{\f{3(\alpha+2)}{2-\alpha} }}{(\eta-\xi)^{\f{8\alpha}{(2-\alpha)}}}
\B(1+\f{\eta^{2}}{(\eta-\xi)^{2}}\B)^{ \f{3\alpha-2}{2-\alpha}}
\|u\|_{L^{p}(T-\eta^{2},T;\dot{\mathcal{M}}^{q,1}(B(\eta)))}^{\f{4}{2-\alpha}} +\f{1}{6}E(\eta).\label{last2}
\end{align}
In light of the definition of Morrey spaces, we see that
$$\|\pi\|_{L^{1}(B(\eta))}\leq C\eta^{3-\f{3}{m}} \|\pi\|_{\dot{\mathcal{M}}^{m,1}(B(\eta))}.$$
Using the Young inequality again, we conclude that
\be\ba IV  & \leq   \f{C\eta^{3 }}{(\eta-\xi)^{8}} \|  \pi\|^{2}_{L^{1} (Q(\eta))}+\f{1}{6}E(\eta)\\
& \leq \f{C\eta^{9-\f6m }}{(\eta-\xi)^{8}} \|  \pi\|^{2}_{L^{1}(T-\eta^{2},T;\dot{\mathcal{M}}^{m,1}(B(\eta)))} +\f{1}{6}E(\eta).
\label{last1}
\ea\ee
After plugging  \eqref{last2}-\eqref{last1} into \eqref{last3}, we apply
  the  classical Iteration Lemma  \cite[Lemma V.3.1,   p.161]{[Giaquinta]} to finish the proof.
\end{proof}

\begin{proof}[Proof of Corollary \ref{the1.3}]
Recall the $\varepsilon$-regularity criteria below shown in \cite{[HWZ]}: $u\in L^{\infty}(Q(1/2)) $ provided that
\be\label{j0il1}\ba
&\|u\|_{ L^{p }(-1,0;L^{q}(B(1))) }+\|\pi\|_{ L^{1}(Q(1)) }\leq \varepsilon_{3},\ea\ee
where
$$1\leq 2/p+3/q <2 ~~ \text{and }~~1< p<\infty.$$
Note that for any $m\geq 1$,
$$
\|\pi\|_{L^{1}(B(r))}\leq Cr^{3-\f{3}{m}} \|\pi\|_{\dot{\mathcal{M}}^{m,1}(B(r))}.
$$
This together with
the energy  bound \eqref{key ineq}   and
\eqref{j0il1} means \eqref{0il1}.
\end{proof}
\subsection{Proof of Theorem \ref{the1.1}  }

\begin{proof}
Let $\lambda=[2a(T-t)]^{-1/2}$. By means of an elementary change of variables and the definition of Morrey spaces, we have
$$\ba
R^{3(\f{1}{q}-\f{1}{l})}\B(\int_{B_{x}(R)\cap B_{x_{0}}(r)}|u(y,t)|^{l}dy\B)^{\f1l}&\leq   R^{3(\f{1}{q}-\f{1}{l})}\B(\int_{B_{x}(R) }|u|^{l}dy\B)^{\f1l}\\
&=   R^{3(\f{1}{q}-\f{1}{l})}\B(\int_{B_{x}(R) }|\lambda U(\lambda y)|^{l}dy\B)^{\f1l}\\
&\leq\lambda^{1-\f3q}\|U\|
_{\mathcal{\dot{M}}^{q,l}(\mathbb{R}^{3})}.
\ea$$
Hence, it follows from \eqref{pm} that for any $r>0$,
 \be\ba\label{3.14}
&\| u\|_{\mathcal{\dot{M}}^{q,l}(B_{x_{0}}(r))}\leq
 \lambda^{1-\f3q}\|U\|
_{\mathcal{\dot{M}}^{q,l}(\mathbb{R}^{3})},
\\
&\| \pi\|_{\mathcal{\dot{M}}^{q/2,l/2}(B_{x_{0}}(r))}\leq
 \lambda^{2-\f{6}q}\|\widetilde{\Pi}\|
_{\mathcal{\dot{M}}^{q/2,l/2}(\mathbb{R}^{3})}\leq C\lambda^{2-\f{6}q}\|U\|^{2}
_{\mathcal{\dot{M}}^{q,l}(\mathbb{R}^{3})}.
\ea \ee
Recall that $\lambda(t)= (2a(T-t))^{-1/2}$ and choose $p>1$ such that $2-p<p-\f{3p}q<2$, then there hold
 \be\ba
&\int_{T-r^{2}}^{T}\| u\|^{p}_{\mathcal{\dot{M}}^{q,l}(B_{x_{0}}(r))}dt\leq
\int_{T-r^{2}}^{T}\lambda^{p-\f{3p}q}\|U\|^{p}
_{\mathcal{\dot{M}}^{q,l}(\mathbb{R}^{3})} dt<\infty,\\
&\int_{T-r^{2}}^{T}\| \pi\|_{\mathcal{\dot{M}}^{q/2,l/2}(B_{x_{0}}(r))}dt\leq
C\int_{T-r^{2}}^{T}\lambda^{2-\f{6}q} \|U\|^{2}
_{\mathcal{\dot{M}}^{q,l}(\mathbb{R}^{3})} dt<\infty.
\ea
\ee
At this stage, the proof of Theorem \ref{the1.1} follows at once from Proposition \ref{the1.2} and  \eqref{local-energy}.
\end{proof}

\section{Proof of Theorem \ref{the1.4} }
We divide the proof of Theorem \ref{the1.4} into three steps. In
Step 1, utilizing the local suitable weak solutions and local Meyer-Gerard-Oru interpolation
inequality \eqref{keyineq}, we establish a new Caccioppoli type inequality. Step 2 is devoted to constructing the pressure $\Pi$ via  $A_p$ weighted inequalities. In the last step, an application of local energy inequalities leads to the proof of
Theorem \ref{the1.4}.

In what follows, we set
$$\|\cdot\| _{L^{p}\dot{\mathcal{M}}^{q,l}(Q(r))}:=\|\cdot\| _{L^{p}(T-r^{2},T;\dot{\mathcal{M}}^{q,l}(B(r)))}.$$
\subsection{Caccioppoli type inequality}
\begin{prop}\label{cip}
Assume that $u$ is a local suitable weak solution to the Navier-Stokes equations \eqref{NS}. Then there holds for any $q>3/2$,
  $$\ba&\|u\|^{2}_{\wred{L^{3}L^{\f{18}{5}}}(Q(\f{R}{2}))}+ \|\nabla u\|^{2}_{L^{2}(Q(\f{R}{2}))}\\
\leq &
  CR^{\f{5q-12}{3q}}\| u\|^{2}_{L^{\f{6q}{2q-3}}\dot{\mathcal{M}}^{q,1}(Q(R))}
  +CR^{\f{4q-15}{2q-3}}\| u\|^{\f{6q}{2q-3}}_{L^{\f{6q}{2q-3}}\dot{\mathcal{M}}^{q,1}(Q(R))}
  +CR^{\f{2q-6}{q}}\| u\|^{3}_{L^{\f{6q}{2q-3}}\dot{\mathcal{M}}^{q,1}(Q(R))}.
\ea$$
 \end{prop}
 As a straightforward consequence of the above proposition, we establish a  regularity criterion at one scale without pressure in Morrey spaces for local suitable weak solutions to \eqref{NS}. The H\"older inequality together with any result of \cite{[CW],[Wolf1],[Wolf2],[JWZ],[WWZ]} and this proposition  yields the desired.
 \begin{coro}\label{the1.5}
		Let  the pair $(u, \pi)$ be a local suitable weak solution to the 3D Navier-Stokes system \eqref{NS} in $Q(1)$.
		There exists an absolute positive constant $\varepsilon_{2}$
		such that if the pair $(u,\pi)$ satisfy
\be\label{0il}\ba
\int_{-1}^{0}\|u\|^{\f{6q}{2q-3}}_{\mathcal{\dot{M}}^{q,1}(B(1))}dt \leq \varepsilon_{2} ~~\text{with} ~~ q>3/2,
\ea\ee
	 then  $u\in L^{\infty}(Q(1/2)).$
	\end{coro}
\begin{proof}[Proof of Proposition \ref{cip}]
 Consider  $0<R/2\leq r<\f{3r+\rho}{4}<\f{r+\rho}{2}<\rho\leq R$. Let $\phi(x,t)$ be a non-negative smooth function supported in $Q(\f{r+\rho}{2})$ such that
$\phi(x,t)\equiv1$ on $Q(\f{3r+\rho}{4})$,
$|\nabla \phi| \leq  C/(\rho-r) $ and $
|\nabla^{2}\phi|+|\partial_{t}\phi|\leq  C/(\rho-r)^{2} .$

Let $\nabla\pi_{h}=\mathcal{W}_{3,B(\f{r+3\rho}{4})}(u)$, then,
from \eqref{ph}-\eqref{p2}, we have
\begin{align}
&\|\nabla \pi_{h}\|_{L^{3}(Q(\f{r+3\rho}{4}))}\leq C\|u\|_{L^{3}(Q(\f{r+3\rho}{4}))},\label{wp1}\\
 &\|  \pi_{1}\|_{L^{2}(Q(\f{r+3\rho}{4}))}\leq C\|\nabla u\|_{L^{2}(Q(\f{r+3\rho}{4}))},\label{wp2}\\
 &\|  \pi_{2}\|_{L^{\f{3}{2}}(Q(\f{r+3\rho}{4}))}\leq C\|  |u|^{2}\|_{L^{\f{3}{2}}(Q(\f{r+3\rho}{4}))}.\label{wp3}
 \end{align}
Since $v=u+\nabla\pi_{h}$, the H\"older inequality and \eqref{wp1} allow us to write
\begin{align}
 \iint_{Q(\rho)} \wred{| v|^2}\Big|  \Delta \phi^{4}+  \partial_{t}\phi^{4}\Big|    \leq& \f{C}{(\rho-r)^{2}}\iint_{Q(\f{r+\rho}{2})} \B(|u|^{2}+|\nabla\pi_{h}|^{2}\B)
\nonumber\\\leq& \f{C\rho^{5/3}}{(\rho-r)^{2}}\B(\iint_{Q(\f{r+\rho}{2})} |u|^{3}+|\nabla\pi_{h}|^{3}\B)^{\f{2}{3}}
\nonumber\\\leq& \f{C\rho^{5/3}}{(\rho-r)^{2}}\|u\|_{L^{3}(Q(\f{r+3\rho}{4}))}^{2}.\label{53.2} \end{align}
 The H\"older  inequality, $v=u+\nabla\pi_{h}$ and  \eqref{wp1}  ensure
 \begin{align}
 \left|\iint_{Q(\rho)}|v|^{2}\phi^{3}u\cdot\nabla \phi\right|
\leq \f{C}{(\rho-r)}
 \| u \|^{3}_{L^{3}(Q(\f{r+3\rho}{4}))}.
 \end{align}
 It follows from the interior estimate of harmonic functions \eqref{h1} and \eqref{wp1} that
$$\ba\|\nabla^{2}\pi_{h} \|_{L^{20/7}(Q(\f{r+\rho}{2}))}&\leq
\f{\wred{C} (r+\rho)  }{(\rho-r)^{ 2}}
\|\nabla\pi_{h} \|_{L^{3}(Q( \f{r+3\rho}{4} ))}\\
&\leq
\f{ C\rho  }{(\rho-r)^{ 2}}
\|u \|_{L^{3}(Q( \f{r+3\rho}{4}))},
\ea$$
which in turn implies that
\begin{align}
&\left|\iint_{Q(\rho)} \phi^{4}( u\otimes v :\nabla^{2}\pi_{h} )\right|  \nonumber\\
\leq&
\| v\phi^{2}\|_{L^{3}(Q(\f{r+\rho}{2}))}
\| u \|_{L^{3}(Q(\f{r+\rho}{2}))}\|\nabla^{2}\pi_{h} \|_{L^{3}(Q(\f{r+\rho}{2}))}
\nonumber\\
\leq &\f{ C\rho  }{(\rho-r)^{ 2}}
\|u \|^{3}_{L^{3}(Q( \f{r+3\rho}{4} ))}.
\end{align}
The H\"older inequality, \eqref{wp2} and Young's inequality yield that
\begin{align}
\left|\iint_{Q(\rho)} \phi^{3} \pi_{1}v\cdot\nabla \phi\right|
&\leq \f{C}{(\rho-r)}\| v\|_{L^{2}(Q(\f{r+\rho}{2}))}
\| \pi_{1} \|_{L^{2}(Q(\f{r+\rho}{2}))} \nonumber\\
&\leq \f{C}{(\rho-r)^{2}}\| v\|^{2}_{L^{2}(Q(\f{r+\rho}{2}))}
+\f{1}{16}\| \pi_{1} \|^{2}_{L^{2}(Q(\f{r+3\rho}{4}))} \nonumber\\
&\leq \f{C\rho^{5/3}}{(\rho-r)^{2}}\|u\|_{L^{3}(Q(\f{r+3\rho}{4}))}^{2}+\f{1}{16}\| \nabla u\|^{2}_{L^{2}(Q(\f{r+3\rho}{4}))}.
\end{align}
We deduce by the H\"older inequality  and \eqref{wp3}   that
\begin{align}
\left|\iint_{Q(\rho)} \phi^{3} \pi_{2}v\cdot\nabla \phi\right|
 \leq \f{C}{(\rho-r)}\| v\phi^{2}\|_{L^{3}(Q(\f{r+\rho}{2}))}
\| \pi_{2} \|_{L^{\f{3}{2}}(Q(\f{r+\rho}{2}))} \leq\f{ C }{(\rho-r)}
\|u \|^{3}_{L^{3}(Q(\f{r+3\rho}{4} ))}.\label{locp5}
\end{align}
Inserting \eqref{53.2}-\eqref{locp5} into the local energy inequality \eqref{wloc1}, we arrive at
\begin{align}
\sup_{T-\rho^{2}\leq t\leq T}\int_{B(\rho)}|v\phi^{2}|^2   + \iint_{Q(\rho)}\big|\nabla( v\phi^{2})\big|^2  \leq& \f{C\rho^{5/3}}{(\rho-r)^{2}}\|u\|_{L^{3}(Q(\f{r+3\rho}{4}))}^{2} +\f{ C\rho  }{(\rho-r)^{ 2}}
\|u \|^{3}_{L^{3}(Q( \f{r+3\rho}{4} ))}\nonumber\\&+\f{ C }{(\rho-r)}
\|u \|^{3}_{L^{3}(Q( \f{r+3\rho}{4} ))}+\f{1}{16}\| \nabla u\|^{2}_{L^{2}(Q(\f{r+3\rho}{4}))}.\label{keyl}
  \end{align}
The interior estimate of harmonic functions
  \eqref{h1} and \eqref{wp1} provide the bound
\be\ba\label{6.10}
\|\nabla\pi_{h}\|^{2}_{L^{3}L^{\f{18}{5}}(Q(r))}&\leq \f{Cr^{\f{5}{3}}}{(\rho-r)^{2}}\|\nabla\pi_{h}\|^{2}_{L^{3}(Q(\f{r+3\rho}{4}))}\leq \f{Cr^{\f{5}{3}}}{(\rho-r)^{2}}
\|u\|^{2}_{L^{3}(Q(\f{r+3\rho}{4}))}.
\ea\ee
Combining the triangle inequality, \eqref{keyl} and \eqref{6.10}, we discover that
\be\ba\label{j4.12}
 \|u\|^{2}_{L^{3}L^{\f{18}{5}}(Q(r))} \leq& \|v\|^{2}_{\wred{L^{3}L^{\f{18}{5}}}(Q(r))}+\|\nabla\pi_{h}\|^{2}_{\wred{L^{3}L^{\f{18}{5}}}(Q(r))}\\
\leq& C\B\{\|v\|_{\wred{L^{\infty}L^2}(Q(r))}^{2}+\|\nabla v\|_{L^{2}(Q(r))}^{2}\B\}+\f{Cr^{\f{5}{3}}}{(\rho-r)^{2}}\|u\|^{2}_{L^{3}(Q(\f{r+3\rho}{4}))}\\
\leq & \f{C\rho^{5/3}}{(\rho-r)^{2}}\|u\|_{L^{3}(Q(\f{r+3\rho}{4}))}^{2}+\f{ C\rho  }{(\rho-r)^{ 2}}
\|u \|^{3}_{L^{3}(Q( \f{r+3\rho}{4} ))} \\&+\f{ C }{(\rho-r)}
\|u \|^{3}_{L^{3}(Q( \f{r+3\rho}{4} ))}+\f{1}{16}\| \nabla u\|^{2}_{L^{2}(Q(\f{r+3\rho}{4}))}.
\ea\ee
According to \eqref{h1} and \eqref{wp1} again, we ascertain
$$
 \|\nabla^{2} \pi_{h}\|^{2}_{L^{2}(Q(r))}\leq \f{Cr^{3}}{(\rho-r)^{5}} \|\nabla\pi_{h}\|^{2}_{L^{2}(Q(\f{r+\rho}{2}))}\leq \f{Cr^{3}\rho^{5/3}}{(\rho-r)^{5}} \| u \|^{2}_{L^{3}(Q(\f{r+3\rho}{4}))}.
$$
Owing to  the triangle \wred{inequality and \eqref{keyl}}, we have
\begin{align}
  \|\nabla u\|^{2}_{L^{2}(Q(r))}\leq & \|\nabla v\|^{2}_{L^{2}(Q(r))}+
 \|\nabla^{2} \pi_{h}\|^{2}_{L^{2}(Q(r))}\nonumber\\ \leq &
  \B\{\f{C\rho^{5/3}}{(\rho-r)^{2}}+\f{Cr^{3}\rho^{5/3}}{(\rho-r)^{5}} \B\} \| u \|^{2}_{L^{3}(Q(\f{r+3\rho}{4}))}\nonumber\\&
  +\B\{\f{ C\rho  }{(\rho-r)^{ 2}} +\f{ C }{(\rho-r)}\B\}
\|u \|^{3}_{L^{3}(Q( \f{r+3\rho}{4} ))}+\f{1}{16}\| \nabla u\|^{2}_{L^{2}(Q(\f{r+3\rho}{4}))}. \label{65.11}
\end{align}
Adding \eqref{j4.12} to \eqref{65.11}, we derive that
\begin{align}
   &\|u\|^{2}_{L^{3}L^{\f{18}{5}}(Q(r))}+\|\nabla u\|^{2}_{L^{2}(Q(r))}\\\leq&
  \B\{\f{C\rho^{5/3}}{(\rho-r)^{2}}+\f{Cr^{3}\rho^{5/3}}{(\rho-r)^{5}} \B\} \| u \|^{2}_{L^{3}(Q(\f{r+3\rho}{4}))}\nonumber\\&
  +\B\{\f{ C\rho  }{(\rho-r)^{ 2}} +\f{ C }{(\rho-r)}\B\}
\|u \|^{3}_{L^{3}(Q( \f{r+3\rho}{4} ))}+\f{1}{8}\| \nabla u\|^{2}_{L^{2}(Q(\f{r+3\rho}{4}))}. \label{j65.11}
\end{align}
Thus, the key ingredient is to control $\|u \|^{3}_{L^{3}}$  on the right hand side of the above inequality.
To this end, invoking the H\"older inequality and \eqref{mgoi}, we see that
\be\label{6.12}\ba
\|u\|^{3}_{L^{3}(B(\f{r+3\rho}{4}))}\leq& C \rho^{9(\f13-\f{3}{2q+6})}\|u\|^{3}_{L^{\f{2q+6}{3}}(B(\f{r+3\rho}{4}))}\\
\leq& C \rho^{\f{3(2q-3)}{2q+6}}\|u\|^{\f{18}{2q+6}}_{\dot{H}^{1}(\mathbb{R}^{3})}
\|u\|^{3-\f{18}{2q+6}}_{\dot{\mathcal{M}}^{q,1}(\mathbb{R}^{3})}.
\ea\ee
Choose a cut-off function $\psi\in C_{0}^{\infty}(B(\rho))$ satisfying $0\leq\psi\leq1$, $\psi\equiv1$ in $B(\f{r+3\rho}{4})$ and $|\nabla\psi|\leq \f{C}{|\rho-r|}$.
In view of triangle inequality and the classical Poincar\'e inequality, we deduce that
\be\ba
\|\nabla [(u-\bar{u}_{B(\rho)})\phi]\|_{L^{2}({B(\rho)})}
\leq& \|\phi\nabla u \|_{L^{2}({B(\rho)})}+\|\nabla\phi (u-\bar{u}_{B(\rho)})\|_{L^{2}({B(\rho)})}\\
\leq&   (C+\f{C\rho}{\rho-r}) \|  \nabla u\|_{L^{2}({B(\rho)})} \\
\leq&    \f{C\rho}{\rho-r}  \|  \nabla u\|_{L^{2}({B(\rho)})}.
\ea\ee
We calculate
\be\label{6.142}\ba
\|(u-\bar{u}_{B(\rho)})\phi\|_{\dot{\mathcal{M}}^{q,1}
(\mathbb{R}^{3})}&\leq
\|u\phi\|_{\dot{\mathcal{M}}^{q,1}(B(\rho))}
+\|\phi\bar{u}_{B(\rho)}\|_{\dot{\mathcal{M}}^{q,1}(B(\rho))}\\
&\leq C
\|u\|_{\dot{\mathcal{M}}^{q,1}(B(\rho))}.
\ea\ee
Thanks to the triangle inequality again and \eqref{6.12}-\eqref{6.142}, we infer that$$\ba
\|u\|^{3}_{L^{3}(B(\f{r+3\rho}{4}))}
\leq& C \rho^{\f{3(2q-3)}{2q+6}}
\|u-\bar{u}_{B(\rho)}\|^{3}_{L^{\f{2q+6}{3}}(B(\f{r+3\rho}{4}))}
+\|\bar{u}_{B(\rho)}\|^{3}_{L^{3}(B(\f{r+3\rho}{4}))}\\\leq& C \rho^{\f{3(2q-3)}{2q+6}}
\|\phi(u-\bar{u}_{B(\rho)})\|^{3}_{L^{\f{2q+6}{3}}(B(\rho))}
+\|\bar{u}_{B(\rho)}\|^{3}_{L^{3}(B(\f{r+3\rho}{4}))}\\
\leq& C \rho^{\f{3(2q-3)}{2q+6}}\|(u-\bar{u}_{B(\rho)})\phi\|^{\f{18}{2q+6}}_{\dot{H}^{1}(\mathbb{R}^{3})}
\|(u-\bar{u}_{B(\rho)})\phi\|^{3-\f{18}{2q+6}}_{\dot{\mathcal{M}}^{q,1}(\mathbb{R}^{3})}+C
\rho^{3-\f{9}{q}}\|u\|^{3}_{\dot{\mathcal{M}}^{q,1}(B(\rho))}\\
\leq&    \f{C\rho^{ \f{8q-3 }{2q+6}}}{(\rho-r)}   \|\nabla u\|^{\f{18}{2q+6}}_{L^{2}(B(\rho))}
\|u\|^{3-\f{18}{2q+6}}_{\dot{\mathcal{M}}^{q,1}(B(\rho))}+C
\rho^{3-\f{9}{q}}\|u\|_{\dot{\mathcal{M}}^{q,1}(B(\rho))}^{3}.
\ea$$

From the H\"older inequality, we know that
$$\ba
&\int^{T}_{T-\f{(r+3\rho)^{2}}{16}}\|u\|^{3}_{L^{3}(B(\f{r+3\rho}{4}))}dt\\
\leq& \f{C\rho^{\f{ 8q-3}{2q+6}}}{(\rho-r)} \int^{T}_{T-\f{(r+3\rho)^{2}}{16}}\|\nabla u\|^{\f{18}{2q+6}}_{L^{2}(B(\rho))}
 \|u\|^{\f{6q}{2q+6}}_{{\dot{\mathcal{M}}^{q,1}(B(\rho))}}dt+
C\rho^{3-\f{6}{q}} \int^{T}_{T-\f{(r+3\rho)^{2}}{16}}\|u\|^{3}_{\dot{\mathcal{M}}^{q,1}(B(\rho))}dt\\
\leq& \f{C\rho^{\f{ 8q-3}{2q+6}}}{(\rho-r)}\B(\int^{T}_{T-\f{(r+3\rho)^{2}}{16}}\|\nabla u\|^{2}_{L^{2}(B(\rho))}dt\B)^{\f{9}{2q+6}}
\B(\int^{T}_{T-\f{(r+3\rho)^{2}}{16}}
 \|u\|^{\f{6q}{2q-3}}_{{\dot{\mathcal{M}}^{q,1}(B(\rho))}}dt\B)^{\f{2q-3}{2q+6}}\\&+C
\rho^{3-\f{6}{q}}\B(\int^{T}_{T-\f{(r+3\rho)^{2}}{16}}
\|u\|^{\f{6q}{2q-3}}_{\dot{\mathcal{M}}^{q,1}(B(\rho))}dt\B)^{\f{2q-3}{2q}},
\ea$$
where we have used the fact that $\f{18}{2q+6}<2$ and $3\leq \f{6q}{2q-3}$.

Plugging this inequality into \eqref{j65.11} and using the Young inequality, we obtain
   $$\ba
&\|u\|^{2}_{L^{3}L^{\f{18}{5}}(Q(r))}+ \|\nabla u\|^{2}_{L^{2}(Q(r))}
  \\\leq&
  \B\{\f{C\rho^{5/3}}{(\rho-r)^{2}}+\f{Cr^{3}\rho^{5/3}}{(\rho-r)^{5}} \B\}^{\f{q+3}{q}}\f{\rho^{\f{8q-3}{3q}}}{(\rho-r)^{\f{2q+6}{3q}}} \| u \|^{2}_{L^{\f{6q}{2q-3}}\dot{\mathcal{M}}^{q,1}(Q(\rho))}\\
  &+\B\{\f{C\rho^{5/3}}{(\rho-r)^{2}}+\f{Cr^{3}\rho^{5/3}}{(\rho-r)^{5}} \B\}
  \rho^{\f{2q-4 }{q}}\|u\|^{2}_{L^{\f{6q}{2q-3}}\dot{\mathcal{M}}^{q,1}(Q(\rho))}\\&
  +\B\{\f{ C\rho  }{(\rho-r)^{ 2}} +\f{ C }{(\rho-r)}\B\}^{\f{2q+6}{2q-3}}\f{\rho^{\f{8q-3}{2q-3}}}{(\rho-r)^{\f{2q+6}{2q-3}}}
\|u \|^{\f{6q}{2q-3}}_{L^{\f{6q}{2q-3}}\dot{\mathcal{M}}^{q,1}(Q( \rho))}\nonumber\\
&+\B\{\f{ C\rho  }{(\rho-r)^{ 2}} +\f{ C }{(\rho-r)}\B\}\rho^{\f{3q-6}{q}}\|u \|^{3}_{L^{\f{6q}{2q-3}}\dot{\mathcal{M}}^{q,1}(Q( \rho))}+\f{3}{16}\| \nabla u\|^{2}_{L^{2}(Q(\rho))}.
 \ea $$
Now, we are in a position to apply  the classical \wred{Iteration Lemma} \cite[Lemma V.3.1,   p.161]{[Giaquinta]} to find that
$$\ba
&\|u\|^{2}_{\wred{L^{3}L^{\f{18}{5}}}(Q(\f{R}{2}))}+ \|\nabla u\|^{2}_{L^{2}(Q(\f{R}{2}))}\\
\leq &
  CR^{\f{5q-12}{3q}}\| u\|^{2}_{L^{\f{6q}{2q-3}}\dot{\mathcal{M}}^{q,1}(Q(R))}
  +CR^{\f{4q-15}{2q-3}}\| u\|^{\f{6q}{2q-3}}_{L^{\f{6q}{2q-3}}\dot{\mathcal{M}}^{q,1}(Q(R))}
  +CR^{\f{2q-6}{q}}\| u\|^{3}_{L^{\f{6q}{2q-3}}\dot{\mathcal{M}}^{q,1}(Q(R))}.
\ea$$
  This achieves the proof of \wred{this proposition}.
\end{proof}

\subsection{Proof of Theorem \ref{the1.4}}
\begin{proof}
It follows from \eqref{3.14} and $\f32<q<6$ that for any $r>0$,
$$
\int_{T-R^{2}}^{T}\| u\|^{\f{6q}{2q-3}}_{\mathcal{\dot{M}}^{q,1}(B(r))}dt\leq
\int_{T-R^{2}}^{T}\lambda^{\f{6q}{2q-3}(1-\f{3}q)}\|U\|^{\f{6q}{2q-3}}
_{\mathcal{\dot{M}}^{q,1}(\mathbb{R}^{3})} dt<\infty.
$$
Thus, this together with Proposition 4.1 implies
\be\label{5.15}
\|u\|^{2}_{L^{3}(Q(\f{R}{2}))}+ \|\nabla u\|^{2}_{L^{2}(Q(\f{R}{2}))}<\infty.
\ee
Let $\lambda_{0}=(2aT)^{-1/2}$. Changing the order of integration, we apply \eqref{5.15} to deduce that
\be\ba \label{6.1}
&\iint_{Q(1)}|u|^{3}dxdt=\int_{\mathbb{R}^{3}}|U|^{3}
 \lambda_{0}^{2}\min\big\{|y|^{-2},
 \lambda_{0}^{-2}\big\}dy,\\
 &\iint_{Q(1)}|\nabla u|^{2}dxdt=\int_{\mathbb{R}^{3}}|\nabla U|^{2}
 2\lambda_{0}^{2}\min\big\{|y|^{-1},\lambda_{0}^{-1}\big\}dy,
\ea\ee
which in turn implies that
$$
\|\nabla U\|_{L^{2}(B_{y_{0}}(1))}=o(|y_{0}|^{1/2}) ~~~\text{and}~~~
\|  U\|_{L^{3}(B_{y_{0}}(1))}=o(|y_{0}|^{2/3}), ~~\text{as} ~~|y_{0}|\rightarrow\infty.
$$
Next, to employ the local energy inequality \eqref{loc} and \eqref{local-energy}, we take into account
  recovering  pressure $\Pi$ via \eqref{6.1}.
Observe that \eqref{6.1} yields
\be\label{6.2}
 \int_{\mathbb{R}^{3}}|U|^{3}|y|^{-2}dy<\infty,
\ee
we can define $$\widetilde{\Pi}=\mathcal{R}_{i}\mathcal{R}_{j}(U_{i}U_{j}),$$
where $U = (U_1,U_2,U_3)$ is determined by \eqref{6.2}.
Due to the classical Calder\'on-Zygmund Theorem with $A_{3/2}$ weights, there holds
\be\label{6.14}
\|\widetilde{\Pi}|y|^{-\f{4}{3}}\|_{L^{3/2}(\mathbb{R}^{3})}\leq C\|U_{i}U_{j}|y|^{-\f{4}{3}}\|_{L^{3/2}(\mathbb{R}^{3})}.
\ee
This means the local integrability of $\widetilde{\Pi}$.
Thus, Weyl's lemma guarantees that $\widetilde{\Pi}$ is smooth.

To proceed further, the fact that $U\in W^{1,2}_{\rm loc}(\mathbb{R}^{3})$ and $\int_{\mathbb{R}^{3}}|U|^{3}(1+|y|)^{-2}dy<\infty$ allows us to  revise the proof of Lemma \ref{cp} to show  that
$\widetilde{\Pi}$ satisfies  $-\Delta \Pi=\partial_{i}\partial_{j}(U_{i}U_{j})$  in the distributional sense.
Here, we just prove that $\nabla \Pi=\nabla \widetilde{\Pi}$.

We use the same notations given in the proof of Lemma \ref{cp}.
By integration by parts, we have
$$\ba
\varepsilon^{3}\left|\int_{\mathbb{R}^{3}}\Delta U\varphi(\varepsilon y)dy\right|=&\varepsilon^{5}\left|\int_{\mathbb{R}^{3}} U|y|^{-\f{2}{3}}|y|^{\f{2}{3}}\Delta\varphi(\varepsilon y)dy\right| \\ \leq& \varepsilon^{5}
\B(\int_{|y|<\f{1}{\varepsilon}}|U|^{3}|y|^{-2}dy\B)^{\f13}
\B(\int_{|y|<\f{1}{\varepsilon}}|y||\Delta\varphi(\varepsilon y)|^{\f{3}{2}}dy\B)^{ \f23}\\
\leq& C\varepsilon^{\f{7}{3}} \B(
\int_{\mathbb{R}^{3}}
|U|^{3}|y|^{-2}dy\B)^{\f13}.
\ea$$
Exactly as the above, we get
$$\ba
\varepsilon^{3}\left|\int_{\mathbb{R}^{3}}U\varphi(\varepsilon y)dy\right| &\leq \varepsilon^{3} \B(\int_{|y|<\f{1}{\varepsilon}}|U|^{3}|y|^{-2}dy\B)^{\f13}
\B(\int_{|y|<\f{1}{\varepsilon}}|y||\varphi(\varepsilon y)|^{\f{3}{2}}dy\B)^{ \f23}\\
&\leq C\varepsilon^{\f{1}{3}}\B(
\int_{\mathbb{R}^{3}}
|U|^{3}|y|^{-2}dy\B)^{\f13}.
\ea$$
In view of  integrating
by parts once again, we see that
\be\ba
\varepsilon^3\int_{\mathbb{R}^{3}}(y\cdot\nabla) U(y)\varphi(\varepsilon y) dy = -\varepsilon^3\int_{\mathbb{R}^{3}} 3U(y)\varphi(\varepsilon y)dy- \varepsilon^3 \int_{\mathbb{R}^{3}} U(y)\{(\varepsilon y)\cdot \nabla\varphi(\varepsilon y)\}dy.
\ea\ee
Hence, proceeding as the above, we can also control $\left|\varepsilon^3\int_{\mathbb{R}^{3}}(y\cdot\nabla) U(y)\varphi(\varepsilon y) dy\right|$.

Thanks to the  divergence-free condition, we arrive at
\be\ba
&\varepsilon^3\int_{\mathbb{R}^{3}}(U\cdot\nabla) U(y)\varphi(\varepsilon y) dy\\
=&-\varepsilon^{4}\int_{\mathbb{R}^{3}}  U\otimes U(y)\cdot\nabla\varphi(\varepsilon y)dy,
\ea\ee
which in turn implies that
\be\label{UUe2}\ba
&\varepsilon^3\left|\int_{\mathbb{R}^{3}}(U\cdot\nabla) U(y)\varphi(\varepsilon y) dy\right| \\
\leq& \varepsilon^{4} \B(\int_{|y|<\f{1}{\varepsilon}}|U|^{3}|y|^{-2}dy\B)^{\f23}
\B(\int_{|y|<\f{1}{\varepsilon}}|y|^{4}|\nabla\varphi(\varepsilon y)|^{3}dy\B)^{\f{1}3}\\
\leq& C\varepsilon^{\f{5}{3}}\B(
\int_{\mathbb{R}^{3}}
|U|^{3}|y|^{-2}dy\B)^{\f23}.
\ea\ee
It remains to bound the last term involving pressure $\widetilde{\Pi}$. Integration by parts gives
\be\ba
\varepsilon^3\int_{\mathbb{R}^{3}}\nabla \widetilde{\Pi}(y)\,\varphi(\varepsilon y) dy
  =-\varepsilon^{4}\int_{\mathbb{R}^{3}}  \widetilde{\Pi}(y)\nabla\varphi(\varepsilon y)dy.
\ea\ee
Furthermore, due to (\ref{6.14}), a variant of (\ref{UUe2}) provides the estimate
$$
\varepsilon^{4}\left|\int_{\mathbb{R}^{3}}  \widetilde{\Pi}(y)\nabla\varphi(\varepsilon y)dy\right|
\leq C\varepsilon^{\f{5}{3}}\B(
\int_{\mathbb{R}^{3}}
|U|^{3}|y|^{-2}dy\B)^{\f23}.
$$
Therefore, we get $\nabla\Pi=\nabla \widetilde{\Pi}$.
Then we can define $\pi$ by $\widetilde{\Pi}$ via \eqref{learysb}.
Note that
$$\iint_{Q(1)}|\pi|^{\f32}dxdt=\int_{\mathbb{R}^{3}}|\widetilde{\Pi}|^{\f32}
 \lambda_{0}^{2}\min\big\{|y|^{-2},
 \lambda_{0}^{-2}\big\}dy.$$
From this and \eqref{6.14}, we know that
$\pi\in L^{3/2}(Q(1 ))$. This together with $u\in L^{3}(Q(1 ))$ implies \eqref{local-energy}  via local energy inequality \eqref{loc}, which concludes the proof of Theorem \ref{the1.4}.
\end{proof}

\section*{Acknowledgement}
The authors would like to express their sincere gratitude to Dr. Daoguo Zhou for short discussion
involving  interpolation inequality in Morrey spaces.
Q. Jiu was partially supported by the National Natural Science Foundation of
China (NNSFC) (No. 11671273, No. 11931010), key research project of the Academy for Multidisciplinary Studies of CNU, and Beijing Natural Science Foundation (BNSF) (No. 1192001).
The research of Wang was partially supported by  the National Natural Science Foundation of China under grant (No. 11971446 and  No. 11601492).
The research of Wei was partially supported by the National Natural Science Foundation of China under grant (No. 11601423, No. 11701450, No. 11701451, No. 11771352, No. 11871057) and Scientific Research Program Funded by Shaanxi Provincial Education Department (Program No. 18JK0763).

	\end{document}